%% file: arXiv_after0923.tex
\documentclass[hidelinks,onefignum,onetabnum,final]{siamart220329}
\input{ex_shared}
\begin{document}
	
	\maketitle
	
	\begin{abstract}
It is significant and challenging to solve eigenvalue problems of partial differential operators when many highly accurate eigenpair approximations are required. The adaptive finite element discretization based parallel orbital-updating method, which can significantly reduce the computational cost and enhance the parallel scalability, has been shown to be efficient in electronic structure calculations. However, there is no any mathematical justification for this method in literature. In this paper, we will show the convergence of the method for clustered eigenvalue problems of linear partial differential operators.
\end{abstract}
	
	
	
			\section{Introduction}
	Solving eigenvalue problems efficiently plays a crucial role in many fields of science and engineering. Due to the inherent challenges in the problems considered and the large-scale orthogonalization
	operations required, the computational cost is very high when highly accurate approximations of many eigenpairs are required. Taking electronic structure calculations of large systems \cite{cances2023density, oliveira2020cecam} as an example, we see that the computation involves global orthogonalization operations, which are expensive and  restrict their effective parallelization.
	
	To reduce both the computational cost and the parallelization restriction, a parallel orbital-updating (ParO) approach has been proposed in \cite{dai2014parallel} and developed in \cite{dai2025numerical, dai2021parallel, pan2017parallel} for solving eigenvalue problems or their corresponding energy minimization models resulting from electronic structure calculations. A number of numerical experiments have demonstrated the advantages of the ParO method over the existing methods, in both the computational cost and the parallel scalability. It should be pointed out that the ParO method has been incorporated into  Quantum ESPRESSO, an open-source software for the first-principles calculations \cite{quantum, oliveira2020cecam}.
	
	Adaptive finite element (FE) methods are naturally applied to solve eigenvalue problems of partial differential operators, which allows for better mesh adjustment to accommodate the characteristics of the exact solutions and the model's singularities, ultimately reducing the scale of the computational problem after discretization and enhancing the computational efficiency \cite{dai2015convergence, dai2008convergence, demlow2017convergence,  gallistl2015optimal,garau2011convergence, garau2009convergence, giani2009convergent}.
	The adaptive FE discretization based ParO (adaptive ParO for short) method incorporates the ParO method into adaptive FE discretizations, and can significantly reduce the computational cost and enhance the parallel scalability. The adaptive ParO method, whose details are shown in \Cref{algo:fram_matr} and \Cref{flow_fram}, transfers solving large scale eigenvalue problems directly to solving some independent source problems in adaptive FE spaces and some small scale algebraic eigenvalue problems. It allows the two-level parallelization: one level is solving these independent source problems in parallel intrinsically; the other level is solving each source problem by traditional algebraic parallel strategies or domain decomposition approaches. 
	
	Most recently, we have presented the numerical analysis for the ParO method on a given FE mesh and shown that the ParO method converges rapidly when the appropriate finite dimensional discretization is provided \cite{dai2025numerical}, which implies the necessity and advantages of incorporating the ParO method into adaptive FE discretizations. Indeed, the adaptive ParO method has been successfully applied to electronic structure calculations for several typical molecular systems in \cite{dai2014parallel},  including $\ce{H2O}$ (water), $\ce{C9H8O4}$ (aspirin), $\ce{C5H9O2N}$ ($\alpha$ amino acid), $\ce{C20H14N4}$ (porphyrin), and $\ce{C60}$ (fullerene) systems, which shows that the adaptive ParO method has the great potential for large scale parallelization and can systematically improve the accuracy of approximations. 
	
	In this paper, we will study the convergence of the adaptive ParO method for eigenvalue problems, taking the following eigenvalue problem as an example: find $\lambda\in\mathbb{R}$ and $u\in H_{0}^{1}(\Omega)$ with $\Vert u\Vert_{0, \Omega}=1$ such that
	\begin{align*}
		\mathcal{L}u:=-\nabla\cdot(A\nabla u)+cu=\lambda u,
	\end{align*}
	where $\Omega$, $A$ and $c$ are introduced in Section 2.1. Note that when many eigenpairs are required, both single and multiple eigenvalues are usually involved. As a result, the traditional measure for the eigenfunction approximation errors does not work well. In our analysis, we employ the eigenspace and the distance from one subspace to another (see, e.g., \cite{chen2014adaptive, dai2015convergence}), which requires sophisticated functional analysis. In particular, in addition to applying the existing adaptive FE analysis tools in the literature such as \cite{chen2011finite, dai2015convergence, dai2008convergence, demlow2017convergence, gallistl2015optimal}, we also need to develop some new approach for analyzing the adaptive ParO method based on the quasi-orthogonality introduced in \cite{dai2025numerical}.
	
	We mention that the philosophy behind the adaptive ParO method is to avoid the direct solution of large scale eigenvalue problems by solving source problems intrinsically in parallel in adaptive FE spaces to produce quasi-orthogonal approximations of orbitals/eigenfunctions, and solving in the lowest-dimensional subspace generated from the approximations. Our numerical analysis employs both computable and theoretical error estimators for both the ParO approximations and the FE approximations. By investigating the relationships among different error estimators, we prove that for the error estimators $\tilde{\eta}_{h}^{2}(\cdot,\Omega)$ and $\eta_{h}^{2}(\cdot, \Omega)$ of the ParO approximations and the FE approximations, respectively, there hold (see \Cref{prop:eqqqqqq} and \Cref{prop:error_est}):
	\begin{align*}
		\tilde{\eta}_{h}^{2}\left(\tilde{U}, \Omega\right)\cong\tilde{\eta}_{h}^{2}\left(\tilde{\mathcal{E}}U, \Omega\right)\approx\eta_{h}^{2}\left(\mathcal{E}_{h}U, \Omega\right)\cong\eta_{h}^{2}\left(U^{h}, \Omega\right)
	\end{align*}
	when the mesh size $h\ll1$ and the ParO approximations reach a certain accuracy. Here,  $\tilde{U}$, $U$ and $U^{h}$ are the vectors consisting of solutions produced by the ParO method, solutions of \cref{eq:weak_form_leq} and solutions of \cref{eq:fd_weak_form_leq}, and $\tilde{\mathcal{E}}U$ and $\mathcal{E}_{h}U$ are the spectral projections of $U$ onto the approximate eigenspace $\bigoplus_{i=1}^{q}\tilde{M}(\lambda_{i})$ produced by the ParO method and the approximate eigenspace $\bigoplus_{i=1}^{q}M_{h}(\lambda_{i})$ spanned by solutions of  \cref{eq:fd_weak_form_leq}, respectively (see Sections 3 and 4 for more details). Consequently, the error estimating for the adaptive ParO approximations can be transformed into the error estimating for adaptive FE approximations, and the contraction property of the adaptive ParO approximations is then obtained (see \Cref{thm:iter_con}). The quasi-orthogonality introduced in \cite{dai2025numerical} plays a crucial role in understanding the philosophy behind the ParO method and analyzing the approximations produced by the ParO method. We apply
	the quasi-orthogonality to derive the convergence of approximations produced by some practical adaptive ParO algorithms (taking \Cref{algo:pou_n_shifted} as an example): if the initial mesh is fine enough and some proper initial values are provided, after several ParO iterations, the error between the ParO approximations and the FE approximations can be controlled by the error estimator for the FE discretization in each refinement (see \Cref{prop:paro_conver_given}). More precisely, for $0<\tilde{\rho}\ll1$ and $n\geqslant0$, after several ParO iterations, there holds
	\begin{align*}
		\sum_{i=1}^{q}\left(\operatorname{dist}^{2}_{a}\left(M_{h_{n}}(\lambda_{i}), \tilde{M}_{n}(\lambda_{i})\right)+\sum_{j=1}^{d_{i}}\left|\lambda_{ij}^{h_{n}}-\lambda_{ij}^{(n)}\right|^{2}\right)\leqslant&\tilde{\rho}\eta_{h_{n}}^{2}\left(\mathcal{E}_{h_{n}}U, \Omega\right).
	\end{align*}
	Consequently, there exists a constant $\beta\in(0,1)$ that is independent of the mesh size $h_{n} (n\geqslant0)$ such that (see \Cref{thm:paro_conver})
	\begin{align*}
		\sum_{i=1}^{q}\left(\operatorname{dist}_{a}^{2}\left(M(\lambda_{i}), \tilde{M}_{n}(\lambda_{i})\right)+\sum_{j=1}^{d_{i}}\left|\lambda_{ij}^{(n)}-\lambda_{i}\right|\right)\lesssim \beta^{2n}, \quad \forall n\geqslant0.
	\end{align*}
	Our analysis shows that to ensure the convergence of the approximations, theoretically, we only need to match the errors of approximations produced by the ParO iterations to the approximation accuracy of FE discretizations. When the discretization is not accurate enough, we do not need to perform the ParO iterations too many times, which will significantly enhance the computational efficiency of the algorithms.
	
	The rest of this paper is organized as follows. We introduce the relevant notation and some existing results of a model problem and its FE discretization in Section 2. We present the framework of the adaptive ParO method and several practical algorithms in Section 3. In Section 4, we carry out the numerical analysis of the method, including the relationship among various types of error estimators, the contraction property of the adaptive ParO approximations, and the convergence of approximations produced by the practical adaptive ParO algorithm. We give concluding remarks in Section 5. Finally, we provide several detailed proofs in \Cref{appen}.
	
		\section{Preliminaries}In this section, we introduce some notation and recall several existing results for eigenvalue problems and their FE approximations that will be used in our discussions.
	\subsection{Model problem}
	Let $\Omega\subset\mathbb{R}^{d}(d\geqslant1)$ be a polygonal domain. We shall use the standard notation for Sobolev spaces with associated norms (see, e.g. \cite{adams2003sobolev, ciarlet1990handbook}). Let $H_{0}^{1}(\Omega)=\{v\in H^{1}(\Omega):v|_{\partial\Omega}=0\}$ and $(\cdot, \cdot)$ be the standard $L^{2}$ inner product. For convenience, the symbol $\lesssim$ will be used. The notation $B_{1}\lesssim B_{2}$ means that $B_{1}\leqslant CB_{2}$ for some generic constant $C$ that is independent of mesh parameters. The notation $\gtrsim$ is defined similarly, and the notation $B_{1}\cong B_{2}$ means $B_{1}\lesssim B_{2}$ and $B_{1}\gtrsim B_{2}$.
	
	We consider the following eigenvalue problem of a second-order elliptic partial differential operator: find $\lambda\in\mathbb{R}$ and $0\ne u\in H_{0}^{1}(\Omega)$ such that  
	\begin{align}\label{source_pro}  
		\mathcal{L}u:=-\nabla\cdot(A\nabla u)+cu=\lambda u,  
	\end{align}  
	where the coefficient matrix $A:\Omega\rightarrow\mathbb{R}^{d\times d}$ is symmetric positive definite with the smallest eigenvalue uniformly bounded away from $0$ and piecewise Lipschitz continuous, and $0\leqslant c\in L^{\infty}(\Omega)$. The corresponding weak form reads that: find $\lambda\in\mathbb{R}$ and $0\neq u\in H_{0}^{1}(\Omega)$ satisfying  
	\begin{align}\label{eq:weak_form_leq}  
		a(u,v)=\lambda b(u,v),\quad \forall v\in H_{0}^{1}(\Omega),  
	\end{align}  
	with bilinear forms defined by  
	\begin{align*}  
		a(u,v)=(A\nabla u, \nabla v)+(cu,v),\quad b(u,v)=(u,v), \quad\forall u,v\in H^1_0(\Omega).
	\end{align*}  
	
	We assume that  
	\begin{align}\label{ine:as1}  
		a(v,w)\lesssim\| v\|_{1, \Omega}\| w\|_{1, \Omega}, \quad\forall v, w\in H_{0}^{1}(\Omega)
	\end{align}  
	and  
	\begin{align}\label{ine:as2}  
		a(v,v)\gtrsim\|v\|_{1,\Omega}^{2}, \quad\forall v\in H_{0}^{1}(\Omega).  
	\end{align}  
	We mention that the derived results remain valid for a broader class of operators satisfying 
	\begin{align*}  
		\|v\|^2_{1, \Omega}-C\|v\|^{2}_{0, \Omega}\lesssim a(v,v),\quad\forall v\in H^1_0(\Omega),  
	\end{align*}  
	where $C>0$ is some constant (see Remark 2.9 in \cite{dai2008convergence} for details). An illustrative example is provided by the Schrödinger equation with $c$ being the Coulomb potential.  
	
	The eigenvalue problem \eqref{eq:weak_form_leq} admits a sequence of increasing eigenvalues $0<\lambda_{1}<\lambda_{2}<\cdots$ with finite multiplicities. Let $d_{i}$ represent the multiplicity of $\lambda_{i}$. The eigenvalues indexed as $i1,\ldots, id_{i}$ follow the ordering:  
	\begin{align*}  
		\lambda_{i-1}<\lambda_{i}=\lambda_{i1}=\cdots=\lambda_{id_{i}}<\lambda_{i+1}, \quad i=1,2,\ldots.  
	\end{align*}
	The corresponding eigenfunctions $\{u_{ij}\}_{j=1}^{d_{i}}$ satisfy $b(u_{ij}, u_{kl})$ $=\delta_{ik}\delta_{jl}$, where $\delta_{ik}$ and $\delta_{jl}$ denote the Kronecker delta. The eigenspace associated with $\lambda_{i}$ is denoted by $M(\lambda_{i})=\operatorname{span}\{u_{i1}, \ldots, u_{id_{i}}\}$. For $1\leqslant j\leqslant d_{i}$ and $1\leqslant s\leqslant d_{r}$, the ordering $(i,j)<(r,s)$ holds when $i<r$, or $i=r, j<s$.
	
	In this paper, we look for the smallest $N$ eigenvalues and their corresponding eigenfunctions of \cref{eq:weak_form_leq}. We require that there exists $q\in\mathbb{N}_{+}$ such that $ \sum_{i=1}^{q}d_{i}=N.$  
	Since both simple and multiple eigenvalues are involved, we introduce the distance with respect to $\|\cdot\|$ from a nontrivial subspace $X\subset H_{0}^{1}(\Omega)$ to another subspace $Y\subset H_{0}^{1}(\Omega)$ as follows (\cite{chatelin2011spectral, kato2013perturbation, knyazev2006new}),
	\begin{align*}
		\operatorname{dist}(X, Y)=\sup_{u\in X, \Vert u\Vert=1}\inf_{v\in Y}\Vert u-v\Vert.
	\end{align*}
	Consistently, for any $0\neq u,v\in H_{0}^{1}(\Omega)$, the distance $\operatorname{dist}(u,v)$  is defined by the distance from the one-dimensional subspace spanned by $u$ to the one spanned by $v$, that is
	\begin{align*}
		\operatorname{dist}(u,v)=\operatorname{dist}\left(\operatorname{span}\{u\}, \operatorname{span}\{v\}\right).
	\end{align*}
	$\operatorname{dist}(u,v)$ is actually the sine of the angle between $u$ and $v$, and is independent of the norms of vectors. Moreover, if $\operatorname{dim}(X)=\operatorname{dim}(Y)<\infty$, there holds
	\begin{align*}
		\operatorname{dist}(X, Y)=\operatorname{dist}(Y, X).
	\end{align*}    
	For convenience, we shall use the notation $\operatorname{dist}_{a}(\cdot, \cdot)$ and $\operatorname{dist}_{b}(\cdot, \cdot)$ when $\Vert\cdot\Vert$  is replaced by $\Vert\cdot\Vert_{a}:=\sqrt{a(\cdot, \cdot)}$ and $\Vert\cdot\Vert_{b}:=\sqrt{b(\cdot, \cdot)}$, respectively. 
	
	The following lemma will be used in our analysis, whose proof is provided in \Cref{proof:lem:dissyysybs}.
	\begin{lemma}\label{lem:dissyysybs}
		Let $X, Y\subset H_{0}^{1}(\Omega)$. If $\{x_{i}\}_{i=1}^{d}$ is an orthogonal basis with respect to $a(\cdot, \cdot)$, i.e., 
		\begin{align*}
			a(x_{i}, x_{j})=0,\quad i\neq j,
		\end{align*}
		then 
		\begin{align*}
			\operatorname{dist}_{a}^{2}(x, Y)\leqslant\sum_{i=1}^{d}\operatorname{dist}_{a}^{2}(x_{i}, Y),\quad x\in X.
		\end{align*}
	\end{lemma}
	
	\subsection{Finite element discretization}
	Let $\{\mathscr{T}_{h}\}$ be a shape regular family of conforming FE meshes over $\Omega$, i.e., there exists a constant $\gamma^{*}$ such that 
	\begin{align*}
		\frac{h_{T}}{\rho_{T}}\leqslant\gamma^{*},\quad \forall T\in\bigcup_{h}\mathscr{T}_{h},
	\end{align*}
	where for each $T\in\mathscr{T}_{h}$, $h_{T}$ is the diameter of $T$, and $\rho_{T}$ is the diameter of the biggest ball contained in $T$, $h=\max\{h_{T}: T\in\mathscr{T}_{h}\}$. Denote the set of interior faces (edges or sides) of $\mathscr{T}_{h}$ by $\mathscr{E}_{h}$. Let 
	\begin{align*}
		S^{h}(\Omega)=\{v\in C(\bar{\Omega}):v|_{T}\in P_{T}^{k},\forall T\in\mathscr{T}_{h}\},
	\end{align*}
	where $P_{T}^{k}$ is the space of polynomials whose degree is not greater than a positive integer $k$. Set $V^{h}:=S^{h}(\Omega)\cap H_{0}^{1}(\Omega)$ with $\operatorname{dim}(V^{h})=N_{g}>N$. 
	
	The standard FE discretization of \cref{eq:weak_form_leq} reads as follows: find $\lambda^{h}\in\mathbb{R}$ and $0\neq u^{h}\in V^{h}$ such that
	\begin{align}\label{eq:fd_weak_form_leq}
		a(u^{h},v)=\lambda^{h} b(u^{h},v),\quad \forall v\in V^{h}.
	\end{align}
	We see that $N_{g}> N$ and may order the eigenvalues of \cref{eq:fd_weak_form_leq} as follows:
	\begin{align*}
		0<\lambda^{h}_{11}\leqslant\cdots\leqslant\lambda^{h}_{1d_{1}}\leqslant\cdots\leqslant\lambda^{h}_{q1}\leqslant\cdots\lambda^{h}_{qd_{q}}\leqslant\cdots.
	\end{align*}
	Suppose the corresponding eigenfunctions $\{u^{h}_{ij}\}_{(1,1)\leqslant(i,j)\leqslant(q,d_{q})}$ satisfy that $b(u^{h}_{ij}, u^{h}_{kl})=\delta_{ik}\delta_{jl}$. For $i=1,\ldots,q$, set $M_{h}(\lambda_{i})=\operatorname{span}\{u^{h}_{i1}, \ldots, u^{h}_{id_{i}}\}$. 
	
	Define the operator $K_{h}:L^{2}(\Omega)\rightarrow V^{h}$ by
	\begin{align*}
		a(K_{h}w, v)=b(w, v), \quad\forall w\in L^{2}(\Omega), \quad\forall v\in V^{h}.
	\end{align*}
	We see that eigenvalues of $K_{h}$ are $\{(\lambda_{ij}^{h})^{-1}\}$. 
	Let $\Gamma_{i}$ be a circle in the complex plane centred at $\lambda_{i}^{-1}$ and not enclosing any other $\lambda_{j}^{-1} (j\neq i)$. For sufficiently small $h$, there is no other eigenvalues of $K_{h}$ contained in $\Gamma_{i}$ except $(\lambda_{i1}^{h})^{-1}, \ldots, (\lambda_{id_{i}}^{h})^{-1}$. The spectral projection associated with $K_{h}$ and $\lambda^{h}_{i1}, \ldots, \lambda^{h}_{id_{i}}$ is defined as 
	\begin{align}\label{def:spec_proj}
		E_{h}^{(i)}:=E_{h}^{(i)}(\lambda_{i})=\frac{1}{2\pi\mathrm{i}}\int_{\Gamma_{i}}(z-K_{h})^{-1}dz,\quad i=1,\ldots,q.
	\end{align}
	It has been proved that $E_{h}^{(i)}(\lambda_{i}):M(\lambda_{i})\rightarrow M_{h}(\lambda_{i})$ is one-to-one and onto when $h$ is sufficiently small (see, e.g., \cite{babuvska1991eigenvalue, babuvska1989finite}). 
	
	The following result is classical and can be found in \cite{babuvska1989finite, chatelin2011spectral, knyazev1985sharp}. 
	\begin{proposition}\label{thm:cluster_eigenfunc}
		If $h\ll1$, then
		\begin{align*}
			\left\Vert u_{ij}-E_{h}^{(i)}u_{ij}\right\Vert_{a}\lesssim\operatorname{dist}_{a}(M(\lambda_{i}), V^{h}),\quad \forall (1,1)\leqslant(i,j)\leqslant(q,d_{q}).
		\end{align*}
	\end{proposition}
	
	The following lemma can be found in \cite{d2018optimization, knyazev1985sharp}.
	
	\begin{lemma}\label{thm:cluster_eigen}
		For the FE discretization \cref{eq:fd_weak_form_leq},  there holds that
		\begin{align*}
			0\leqslant\lambda^{h}_{ij}-\lambda_{i}\leqslant\lambda^{h}_{ij}\operatorname{dist}_{a}^{2}(\bigoplus_{i=1}^{q}M(\lambda_{i}), V^{h}),\quad \forall (1,1)\leqslant(i,j)\leqslant(q,d_{q}).
		\end{align*}
	\end{lemma}
	
	Since several eigenspaces are involved, the following notation will be used in our analysis. Define
	\begin{align*}
		\mathcal{M}=\begin{pmatrix}
			M(\lambda_{1}), \cdots, M(\lambda_{q})
		\end{pmatrix}\subset \left(H_{0}^{1}(\Omega)\right)^{N},\quad
		\mathcal{M}_{h}=\begin{pmatrix}
			M_{h}(\lambda_{1}), \cdots, M_{h}(\lambda_{q})
		\end{pmatrix}\subset \left(V^{h}\right)^{N},
	\end{align*}
	and the operator
	\begin{align}\label{eq:oplus_spec_fd}
		\mathcal{E}_{h}:=E_{h}^{(1)}\oplus\cdots\oplus E_{h}^{(q)}: \mathcal{M}\rightarrow\mathcal{M}_{h}
	\end{align}
	by 
	\begin{align*}
		\mathcal{E}_{h}U=\left(E_{h}^{(1)}u_{11}, \ldots, E_{h}^{(1)}u_{1d_{1}}, \ldots, E_{h}^{(q)}u_{q1}, \ldots, E_{h}^{(q)}u_{qd_{q}}\right),
	\end{align*}
	where $U=\left(u_{11}, \ldots, u_{1d_{1}}, \ldots, u_{q1}, \ldots,  u_{qd_{q}}\right)$.
	
	For $W=(w_{1}, \ldots, w_{N})\in \left(H_{0}^{1}(\Omega)\right)^{N}$, we denote
	\begin{align*}
		\Vert W\Vert=\left(\sum_{i=1}^{N}\Vert w_{i}\Vert^{2}\right)^{1/2}.
	\end{align*}
	
		\section{The adaptive ParO method} In this section, we introduce the framework of the adaptive ParO method and present several practical algorithms.
	\subsection{Algorithm framework}
	A framework of the adaptive ParO method for the Kohn-Sham equation, which is an eigenvalue problem of a nonlinear operator, has been proposed in \cite{dai2014parallel}. In this paper, we focus on the eigenvalue problem of linear operators. By adapting Algorithm 3.1 in  \cite{dai2014parallel} to the case of linear operators and further incorporating the description of Algorithm 4.1 in \cite{dai2025numerical}, we can derive the following framework (\Cref{algo:fram_matr}) of the adaptive ParO method for computing the smallest $N$ eigenvalues and their corresponding eigenfunctions of \cref{eq:weak_form_leq}. To better illustrate the incorporation process between the ParO method and adaptive FE discretizations, we provide a flowchart as \Cref{flow_fram}.
	\begin{algorithm}[htbp]
		\caption{Framework for the adaptive ParO method}\label{algo:fram_matr}
		\begin{algorithmic}[1]
			\State Given $\mathscr{T}_{h_{0}}$,  $V^{h_{0}}$ and initial data $\left\{\left(\lambda_{k}^{(-1)}, u_{k}^{(-1)}\right)\right\}_{1\leqslant k\leqslant N}$, let $n=0$;
			\While{not converged with respect to $n$}
			\State Set $\left(\lambda_{k}^{(n, 0)}, u_{k}^{(n, 0)}\right)=\left(\lambda_{k}^{(n-1)}, u_{k}^{(n-1)}\right)$ and $m=0$;
			\While{not converged with respect to $m$}
			\State For $k=1,\ldots, N$, update each orbital $u_{k}^{(n,m)}$ to obtain $u_{k}^{(n,m+1/2)}\in V^{h_{n}}$ in parallel;
			
			\State Let $u_{k}^{(n,m+1)}=u_{k}^{(n,m+1/2)}/\left\Vert u_{k}^{(n,m+1/2)}\right\Vert_{b}$ and $\lambda_{k}^{(n,m+1)}=\left\Vert u_{k}^{(n,m+1)}\right\Vert_{a}^{2}$. Or if necessary, solve \cref{eq:weak_form_leq} in $\tilde{X}_{n,m+1}:=\operatorname{span}\left\{u_{1}^{(n,m+1/2)}, \ldots, u_{N}^{(n,m+1/2)}\right\}$ and obtain eigenpairs $\left(\lambda_{k}^{(n,m+1)}, u_{k}^{(n,m+1)}\right)$; 
			
			\State Let $m = m + 1$;
			\EndWhile
			
			\State Set $\left(\lambda_{k}^{(n)}, u_{k}^{(n)}\right) = \left(\lambda_{k}^{(n,m)}, u_{k}^{(n,m)}\right)$ for $k=1, \ldots, N$;
			
			\State Construct $\mathscr{T}_{h_{n+1}}$ and $V^{h_{n+1}}$ based on an adaptive procedure to
			$\left\{\left(\lambda_{k}^{(n)}, u_{k}^{(n)}\right)\right\}_{1\leqslant k\leqslant N}$;
			
			\State Let $n = n + 1$.
			\EndWhile
		\end{algorithmic}
	\end{algorithm}
	
	\begin{figure}[htbp!]
		\centering
		\includegraphics[width=320pt]{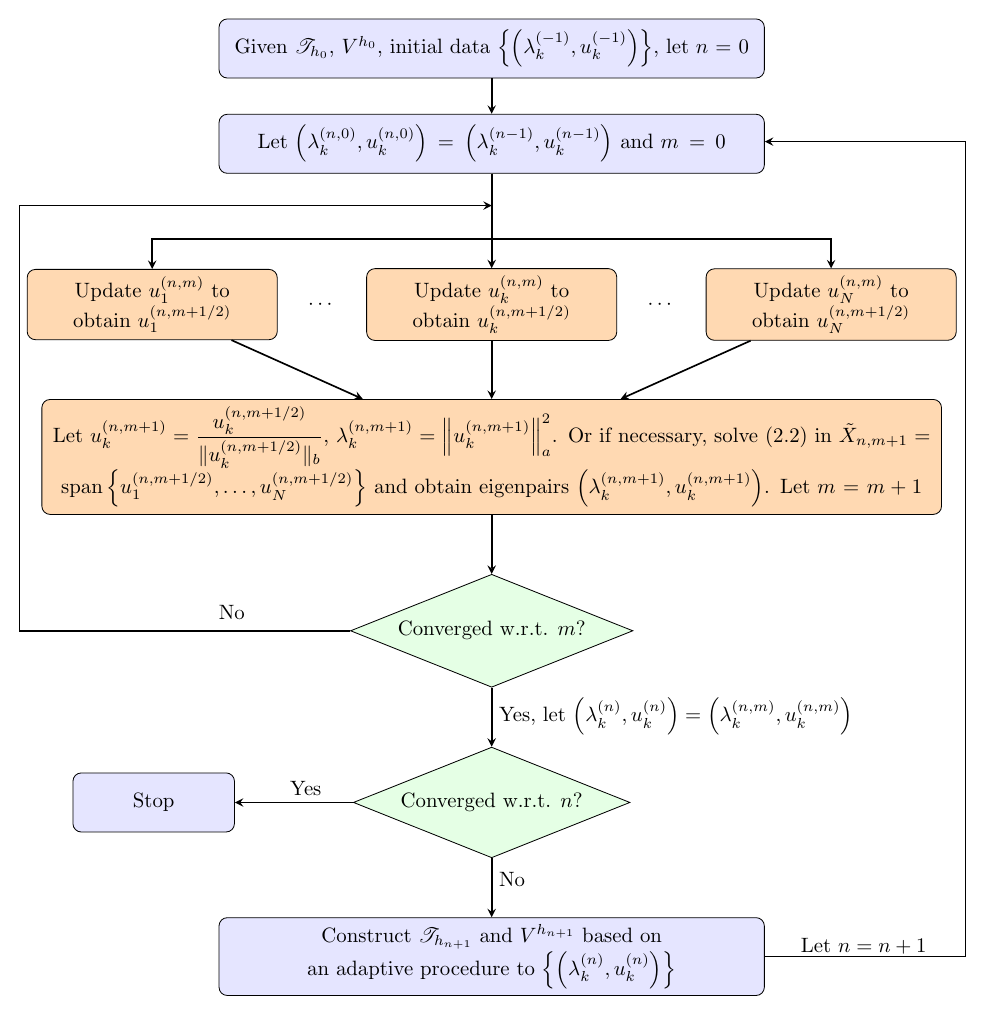}
		\caption{Flowchart of \Cref{algo:fram_matr}}\label{flow_fram}
	\end{figure}
	
	The adaptive ParO method is an orbital/eigenfunction iteration based approach for solving  clustered eigenvalue problems with both single and multiple eigenvalues involved. If proper initial guesses or approximations for the eigenfunctions are provided, then the method will produce approximations to orthonormal bases of the corresponding eigenspaces, more precisely, the corresponding orthonormal eigenfunctions. Details of Step 1 and Step 5 are referred to \cite{dai2014parallel, dai2025numerical}.
	
	As mentioned in \cite{dai2025numerical}, in Step 6 of the $n$-th refinement, we can update the ParO approximations either by simply normalizing the norms or by solving the following small scale eigenvalue problem: find $\left\{\left(\lambda_{k}^{(n, m+1)}, u_{k}^{(n, m+1)}\right)\right\}\subset\mathbb{R}\times \tilde{X}_{n,m+1}$ satisfying $b\left(u_{k}^{(n,m+1)}, u_{l}^{(n,m+1)}\right)=\delta_{kl}$ for $k,l=1,\ldots,N$ and
	\begin{align}\label{eq:pro_ei_pro}
		a\left(u_{k}^{(n,m+1)},v\right)=\lambda_{k}^{(n,m+1)}b\left(u_{k}^{(n,m+1)},v\right)\quad \forall v\in \tilde{X}_{n,m+1}.
	\end{align}

	In the following, we recall how to implement the adaptive FE discretization for eigenvalue problems when the smallest $N$ eigenvalues and the corresponding eigenfunctions are required. An adaptive FE approach usually consists of the following four steps \cite{cascon2008quasi, chen2011finite, dai2015convergence, dai2008convergence}:
	$$
	\mbox{\bf Solve}~\rightarrow~\mbox{\bf Estimate}~\rightarrow~\mbox{\bf Mark}~\rightarrow~\mbox{\bf Refine}.
	$$
	
	\textbf{Solve} step gets the FE approximations in $V^{h_{n}}$ with respect to a given mesh $\mathscr{T}_{h_{n}}$. In \Cref{algo:fram_matr}, the \textbf{Solve} step consists of Steps 4-9.

	\textbf{Estimate} step requires the a posteriori error estimators based on the mesh $\mathscr{T}_{h_{n}}$ and the corresponding output from \textbf{Solve} step. For the sake of clarity, we introduce the following definition of error estimators. We see that they coincide with those proposed in \cite{dai2015convergence}.

 For $\mathscr{T}_{h_{n}}$ and $w\in \operatorname{span}\{u_{1}^{(n)}, \ldots, u_{N}^{(n)}\}$, define the element residual $\tilde{\mathscr{R}}_{T}(\cdot)$ in $T\in\mathscr{T}_{h_{n}}$ and the jump residual $\tilde{\mathscr{J}}_{e}(\cdot)$ on $e \in \mathscr{E}_{h_{n}}$ by 
\begin{align*}
    \tilde{\mathscr{R}}_{T}(w)=&\sum_{k=1}^{N}b\left(w, u_{k}^{(n)}\right)\lambda_{k}^{(n)}u_{k}^{(n)}+\nabla\cdot\left(A\nabla w\right)-cw,\\
    \tilde{\mathscr{J}}_{e}(w)=&-A\nabla w^{+}\cdot\nu^{+}-A\nabla w^{-}\cdot\nu^{-}=\llbracket A \nabla w \rrbracket_e \cdot \nu_e,
\end{align*}
where $e$ is the common side of elements $T^{+}$ and $T^{-}$ with unit outwards normals $\nu^{+}$ and $\nu^{-}$, respectively, and $\nu_{e}=\nu^{-}$. For $T\in\mathscr{T}_{h_{n}}$, define the local error indicator $\tilde{\eta}_{h_{n}}(w, T)$ and $\tilde{\eta}_{h_{n}}(W, T)$ by 
\begin{align*}
    \tilde{\eta}_{h_{n}}^{2}(w, T)&=h_{T}^{2}\left\Vert\tilde{\mathscr{R}}_{T}(w)\right\Vert_{0,T}^{2}+\sum_{e\in\mathscr{E}_{h_{n}}, e\subset\partial T}h_{e}\left\Vert\tilde{\mathscr{J}}_{e}(w)\right\Vert_{0,e}^{2}, \\
    \tilde{\eta}_{h_{n}}^{2}(W,T)&=\sum_{k=1}^{N}\tilde{\eta}_{h_{n}}^{2}(w_{k}, T),
\end{align*}
where $W=(w_{1}, \ldots, w_{N})$. Then the global error estimators $\tilde{\eta}_{h_{n}}(w, \Omega)$ and $\tilde{\eta}_{h_{n}}(W, \Omega)$ are defined by
\begin{align*}
    \tilde{\eta}_{h_{n}}^{2}(w, \Omega)=\sum_{T\in\mathscr{T}_{h_{n}},  T\subset\Omega} \tilde{\eta}_{h_{n}}^{2}(w, T),\quad \tilde{\eta}_{h_{n}}^{2}(W, \Omega)=\sum_{T\in\mathscr{T}_{h_{n}}, T\subset\Omega} \tilde{\eta}_{h_{n}}^{2}(W, T).
\end{align*}

In our calculations, the error estimators $\{\tilde{\eta}_{h_{n}}(U^{(n)}, T)\}_{T\in\mathscr{T}_{h_{n}}}$, which are computable, are used to estimate the error of the ParO approximations in the adaptive refinement of the FE discretizations, where $U^{(n)}=\left(u_{1}^{(n)}, \ldots, u_{N}^{(n)}\right)$. 
	
	\textbf{Mark} step provides a strategy to pick out a subset $\hat{\mathscr{T}}_{h_{n}}$ of $\mathscr{T}_{h_{n}}$ to be refined based on the  posteriori error indicators. The Dörfler marking strategy \cite{dai2015convergence, dorfler1996convergent}, which is commonly used in the adaptive algorithms, is adopted in our algorithms and is  stated as follows.
	
	\vspace{0.5\baselineskip}
	\begin{mdframed}
		\begin{center}
			\textbf{D\"orfler marking strategy}
		\end{center}
		Give a parameter $0 < \theta < 1$.
		
		\begin{enumerate}
			\item Construct a subset $\hat{\mathscr{T}}_{h_{n}}$ of $\mathscr{T}_{h_{n}}$ such that
			\begin{align}\label{ine:dor}
				\sum_{T \in \hat{\mathscr{T}}_{h_{n}}}\tilde{\eta}_{h_{n}}^{2}(U^{(n)},T) \geqslant \theta \tilde{\eta}_{h_{n}}^{2}(U^{(n)},\Omega).
			\end{align}
			\item Mark all the elements in $\hat{\mathscr{T}}_{h_{n}}$.
		\end{enumerate}
	\end{mdframed}
	\vspace{0.5\baselineskip}
	
	\textbf{Refine} step is some iterative or recursive bisection of elements with the minimal refinement condition that marked elements are bisected at least once \cite{cascon2008quasi,maubach1995local}. Given a fixed number $\ell\geqslant1$, for any $\mathscr{T}\in\{\mathscr{T}_{h_{n}}\}_{n\geqslant1}$, and a subset $\hat{\mathscr{T}}\subset\mathscr{T}$ of marked elements, $\mathscr{T}_{*}=\operatorname{\textbf{Refine}}(\mathscr{T}, \hat{\mathscr{T}})$ outputs a conforming mesh $\mathscr{T}_{*}\in \{\mathscr{T}_{h_{n}}\}_{n\geqslant1}$, where at least all elements of $\hat{\mathscr{T}}$ are bisected $\ell$ times. 
	
	The adaptive mesh refining procedure with the Dörfler marking strategy is summarized in \Cref{algo:adaptive_mode}. 
	\begin{algorithm}[htbp]
		\caption{$\operatorname{Adaptive\_Refine}(\left\{\left(\lambda_{k}^{(n)}, u_{k}^{(n)}\right)\right\}, \mathscr{T}_{h_{n}}, \theta)$}\label{algo:adaptive_mode}
		\begin{algorithmic}[1]
			\State Compute local error indicators $\tilde{\eta}_{h_{n}}(u_{k}^{(n)}, T)$ from $\left\{\left(\lambda_{k}^{(n)}, u_{k}^{(n)}\right)\right\}$ for $k=1, \ldots, N$ and $T\in \mathscr{T}_{h_{n}}$;
			\State Construct $\hat{\mathscr{T}}_{h_{n}}\subset\mathscr{T}_{h_{n}}$ by the \textbf{D\"orfler marking strategy} with the parameter $\theta$;
			\State Refine $\mathscr{T}_{h_{n}}$ to get a new conforming mesh $\mathscr{T}_{h_{n+1}}$ by Procedure \textbf{Refine}.
		\end{algorithmic}
	\end{algorithm}
	
	\subsection{Practical algorithms}

	Before presenting the practical algorithms of the adaptive ParO method, we must preprocess the initial data. As noted in the Introduction, we are interested in clustered eigenvalues and their associated eigenfunctions, so we should approximate each eigenspace. We understand that it is not trivial to obtain the multiplicity $d_{i}$ of each eigenvalue $\lambda_{i}$ in advance. One feasible strategy is to cluster the initial guesses $\lambda_{1}^{(0)}\leqslant\lambda_{2}^{(0)}\leqslant\cdots\leqslant\lambda_{N}^{(0)}$ of eigenvalues. Using clustering criteria such as the Bayesian Information Criterion or the silhouette method (see e.g., \cite{rousseeuw1987silhouettes, schwarz1978estimating}), we can obtain $q'$ clusters with $d_{i}'$ eigenpairs in the $i$-th cluster $(i=1,\ldots,q')$, that is,  
\begin{align}\label{eq:clustering_eigenpairs}
    \left\{\left(\lambda_{ij}^{(-1)}, u_{ij}^{(-1)}\right)\right\}_{i=1,\ldots,q',j=1,\ldots,d_{i}'}= \left\{\left(\lambda_{k}^{(-1)}, u_{k}^{(-1)}\right)\right\}_{k=1,\ldots,N}.
\end{align}

		As mentioned in \cite{dai2014parallel}, there are several approaches to update the orbitals in Step 5 of \Cref{algo:fram_matr}, such as the shifted-inverse approach and Chebyshev filtering. We also observe in \cite{dai2025numerical} that one may choose whether to solve the projected eigenvalue problem \cref{eq:pro_ei_pro} in Step 6 or not in the process of ParO iterations. Anyway, there are various available approaches to construct practical algorithms under the framework of the adaptive ParO method (\Cref{algo:fram_matr}). Then, in the $m$-th iteration of the $n$-th refinement,
\begin{align*}
    \left\{\left(\lambda_{ij}^{(n, m)}, u_{ij}^{(n, m)}\right)\right\}_{i=1,\ldots,q',j=1,\ldots,d_{i}'}= \left\{\left(\lambda_{k}^{(n,m)}, u_{k}^{(n,m)}\right)\right\}_{k=1,\ldots,N}
\end{align*}
 are produced as the ParO approximations for the eigenpairs $\left\{\left(\lambda_{i}, u_{ij}\right)\right\}_{i=1,\ldots,q,j=1,\ldots,d_{i}}$.
	
	As a practical example, the shifted-inverse based ParO algorithm applies the shifted-inverse approach to update each orbital and solves a small scale eigenvalue problem in each iteration to obtain the approximations \cite{dai2014parallel, dai2025numerical, pan2017parallel}. The convergence of the algorithm on a given FE mesh is proven in \cite{dai2025numerical}. We incorporate the algorithm into adaptive FE discretizations and restate the shifted-inverse based adaptive ParO algorithm as \Cref{algo:pou_n_shifted} as an example of practical adaptive ParO algorithms.
	\begin{algorithm}[htbp]
		\caption{Shifted-inverse based adaptive ParO algorithm }\label{algo:pou_n_shifted}
		\begin{algorithmic}[1]
			\State Given $\mathscr{T}_{h_{0}}$,  $V^{h_{0}}$, $tol_{1}, tol_{2}>0$ and $0<\theta<1$, provide and cluster initial data by (\ref{eq:clustering_eigenpairs}), i.e., $\left\{\left(\lambda_{ij}^{(-1)}, u_{ij}^{(-1)}\right)\right\}_{(1,1)\leqslant(i,j)\leqslant(q',d'_{q})}$.   Let $n=0$ and $\Delta_{1}=2tol_{1}$;
			\While{$\Delta_{1} > tol_1$}
			
			\State Set $\left(\lambda_{ij}^{(n, 0)}, u_{ij}^{(n, 0)}\right)=\left(\lambda_{ij}^{(n-1)}, u_{ij}^{(n-1)}\right)$, $\bar{\lambda}^{(n,0)}_{i}=\mathcal{C}_i(\{\lambda_{ij}^{(n,0)}\}_{j=1}^{d_{i}'})$ as some convex combination of $\{\lambda_{ij}^{(n,0)}\}_{j=1}^{d_{i}'}$, $m=0$ and $\Delta_{2}=2tol_{2}$;
			\While{$\Delta_{2} > tol_2$}
			\State For $(1,1)\leqslant(i,j)\leqslant(q',d_{q}')$, find $u_{ij}^{(n,m+1/2)}\in V^{h_{n}}$ in parallel by solving   
			\begin{align*}
				a(u_{ij}^{(n,m+1/2)},v)-\bar{\lambda}_{i}^{(n,m)}b(u_{ij}^{(n,m+1/2)},v)=\bar{\lambda}_{i}^{(n,m)}b(u_{ij}^{(n,m)},v)\quad\forall v\in V^{h_{n}};
			\end{align*}
			
			\State Solve the eigenvalue problem: find $\left\{\left(\lambda_{ij}^{(n, m+1)}, u_{ij}^{(n, m+1)}\right)\right\}\subset\mathbb{R}\times \tilde{X}_{n,m+1}:=\mathbb{R}\times\operatorname{span}\left\{u_{11}^{(n,m+1/2)}, \ldots, u_{q'd_{q}'}^{(n,m+1/2)}\right\}$ with
			$b(u_{ij}^{(n,m+1)}, u_{kl}^{(n,m+1)})=\delta_{ik}\delta_{jl}$ satisfying
			\begin{align*}
				a(u_{ij}^{(n,m+1)},v)=\lambda_{ij}^{(n,m+1)}b(u_{ij}^{(n,m+1)},v)\quad \forall v\in \tilde{X}_{n,m+1};
			\end{align*}
			
			\State Set $\Delta_{2} = \frac{\sum_{i=1}^{q'}\sum_{j=1}^{d_{i}'} |\lambda_{ij}^{(n,m+1)} - \lambda_{ij}^{(n,m)}|}{\sum_{i=1}^{q'}\sum_{j=1}^{d_{i}'} |\lambda_{ij}^{(n,m)}|}$, $\bar{\lambda}^{(n,m+1)}_{i}=\mathcal{C}_i(\{\lambda_{ij}^{(n, m+1)}\}_{j=1}^{d_{i}'})$ and $m = m + 1$;
			\EndWhile
			
			\State Set $\left(\lambda_{ij}^{(n)}, u_{ij}^{(n)}\right) = \left(\lambda_{ij}^{(n,m)}, u_{ij}^{(n,m)}\right)$ and $\Delta_{1} =\frac{ \sum_{i=1}^{q'}\sum_{j=1}^{d_{i}'} |\lambda_{ij}^{(n)} - \lambda_{ij}^{(n-1)}|}{\sum_{i=1}^{q'}\sum_{j=1}^{d_{i}'} | \lambda_{ij}^{(n)}|}$;
			
			\State Get $\mathscr{T}_{h_{n+1}}=\operatorname{Adaptive\_Refine}(\left\{\left(\lambda_{ij}^{(n)}, u_{ij}^{(n)}\right)\right\}, \mathscr{T}_{h_{n}}, \theta)$ and construct $V^{h_{n+1}}$;
			\State Let $n = n + 1$.
			\EndWhile
		\end{algorithmic}
	\end{algorithm}
	
If we replace $\bar{\lambda}_i^{(n,m)}$ by some fixed shift in Step 5 and update the ParO approximations simply by normalizing the norms in Step 6 in Algorithm \ref{algo:pou_n_shifted}, we will then obtain a simplified shifted-inverse based ParO algorithm, whose convergence on a given mesh is proven in \cite{dai2025numerical}. Due to space limitations, we will not address more details here.
	
	\section{Numerical analysis}
	In this section, we will carry out the numerical analysis of the adaptive ParO method. By investigating 
	the relationships among the error estimators in Section 4.1, we show that the error estimate for the adaptive ParO approximations can be transformed into the error estimate for the adaptive FE approximations, and obtain the contraction property of the adaptive ParO approximations in Section 4.2. In Section 4.3, we apply the quasi-orthogonality to show that after several ParO iterations, the error between the ParO approximations and the FE approximations can be controlled by the error estimator for the FE discretization. Finally, we prove the convergence of the  approximations produced by some practical adaptive ParO algorithms (taking \Cref{algo:pou_n_shifted} as an example) in Section 4.4. 
	
	In this section, we assume that $q'=q$ and $d_{i}'=d_{i}$ for $i=1,\ldots,q'$. 
	It is noteworthy that, the numerical experiments in \cite{dai2014parallel, pan2017parallel} have shown that such an assumption is unnecessary in practical computations.
	
	\subsection{Relationships among different error estimators}
	For ParO approximations, denote $\tilde{M}_{n}(\lambda_{i})=\operatorname{span}\{u_{i1}^{(n)}, \ldots, u_{id_{i}}^{(n)}\}$ for $i=1,\cdots,q$. For convenience, set
\begin{align*}
\tilde{\mathcal{M}}_{n}=\begin{pmatrix}
\tilde{M}_{n}(\lambda_{1}),\ldots, \tilde{M}_{n}(\lambda_{q})
\end{pmatrix}.
\end{align*}
For $i=1,\ldots,q$, we define $\mathcal{P}^{(i)}$ as the orthogonal projection from $M_{h_{n}}(\lambda_{i})$ onto $\tilde{M}_{n}(\lambda_{i})$ with respect to $a(\cdot, \cdot)$. Then when $\operatorname{dist}_{a}(M_{h_{n}}(\lambda_{i}), \tilde{M}_{n}(\lambda_{i}))<1$, $\mathcal{P}^{(i)}$ is one-to-one and onto. Note that $E_{h_{n}}^{(i)}$ is one-to-one and onto when $h_{n}\ll1$. We obtain that the projection $\tilde{E}_{n}^{(i)}: M(\lambda_{i})\rightarrow \tilde{M}_{n}(\lambda_{i})$ defined by $\tilde{E}_{n}^{(i)}=\mathcal{P}^{(i)}\circ E_{h_{n}}^{(i)}$ is also one-to-one and onto when $h_{n}\ll1$ and $\operatorname{dist}_{a}(M_{h_{n}}(\lambda_{i}), \tilde{M}_{n}(\lambda_{i}))<1$. Similar to \cref{eq:oplus_spec_fd}, define the operator
\begin{align}\label{eq:oplus_spec}
 \tilde{\mathcal{E}}_{n}:=\tilde{E}_{n}^{(1)}\oplus\cdots\oplus \tilde{E}_{n}^{(q)}: \mathcal{M}\rightarrow\tilde{\mathcal{M}}_{n}
\end{align}
by 
\begin{align*}
    \tilde{\mathcal{E}}_{n}U=\left(\tilde{E}_{n}^{(1)}u_{11}, \ldots, \tilde{E}_{n}^{(1)}u_{1d_{1}}, \ldots, \tilde{E}_{n}^{(q)}u_{q1}, \ldots, \tilde{E}_{n}^{(q)}u_{qd_{q}}\right),
\end{align*}
where $U=\left(u_{11}, \ldots, u_{1d_{1}}, \ldots, u_{q1}, \ldots, u_{qd_{q}}\right)$.
In our following theoretical analysis, the error estimators $\{\tilde{\eta}_{h_{n}}(\tilde{\mathcal{E}}_{n}U, T)\}_{T\in\mathscr{T}_{h_{n}}}$, which are uncomputable, will be used. We call $\{\tilde{\eta}_{h_{n}}(\tilde{\mathcal{E}}_{n}U, T)\}_{T\in\mathscr{T}_{h_{n}}}$ the theoretical error estimators for the ParO approximations.
	
	We will also apply the error estimators for the FE approximations in our analysis. Denote by $\left\{\left(\lambda_{ij}^{h_{n}}, u_{ij}^{h_{n}}\right)\right\}$ the smallest $N$ eigenpairs of \cref{eq:fd_weak_form_leq} when $V^{h}=V^{h_{n}}$.
	 Similar to the ParO approximations, for $\mathscr{T}_{h_{n}}$ and $w\in \bigoplus_{i=1}^{q}M_{h_{n}}(\lambda_{i})$, we define the element residual $\mathscr{R}_{T}(\cdot)$ in $T\in\mathscr{T}_{h_{n}}$ and the jump residual $\mathscr{J}_{e}(\cdot)$ on $e \in \mathscr{E}_{h_{n}}$ for the FE approximations by 
\begin{align*}
    &\mathscr{R}_{T}(w)=\sum_{i=1}^{q}\sum_{j=1}^{d_{i}}b\left(w, u_{ij}^{h_{n}}\right)\lambda_{ij}^{h_{n}}u_{ij}^{h_{n}}+\nabla\cdot\left(A\nabla w\right)-c w, \\
    &\mathscr{J}_{e}(w)=\llbracket A \nabla w \rrbracket_e \cdot \nu_e.
\end{align*}
For $T\in\mathscr{T}_{h_{n}}$, define the local error indicator $\eta_{h_{n}}(w, T)$ and $\eta_{h_{n}}(W, T)$ for the FE approximations by 
\begin{align*}
    \eta_{h_{n}}^{2}(w, T)&=h_{T}^{2}\left\Vert\mathscr{R}_{T}(w)\right\Vert_{0,T}^{2}+\sum_{e\in\mathscr{E}_{h_{n}}, e\subset\partial T}h_{e}\left\Vert\mathscr{J}_e(w)\right\Vert_{0,e}^{2}, \\
   \eta_{h_{n}}^{2}(W,T)&=\sum_{i=1}^{q}\sum_{j=1}^{d_{i}}\eta_{h_{n}}^{2}(w_{ij}, T),
\end{align*}
where $W=(w_{11}, \ldots,w_{1d_{1}}, \ldots,w_{q1},\ldots, w_{qd_{q}})$.
 Then the global error estimators $\eta_{h_{n}}(w, \Omega)$ and $\eta_{h_{n}}(W, \Omega)$ are defined by
\begin{align*}
    \eta_{h_{n}}^{2}(w, \Omega)=\sum_{T\in\mathscr{T}_{h_{n}}, T\subset\Omega} \eta_{h_{n}}^{2}(w, T),\quad \eta_{h_{n}}^{2}(W, \Omega)=\sum_{T\in\mathscr{T}_{h_{n}}, T\subset\Omega} \eta_{h_{n}}^{2}(W, T).
\end{align*}

By the above definition, if we set $W=U^{h_{n}}:=\left(u_{11}^{h_{n}}, \ldots, u_{1d_{1}}^{h_{n}}, \ldots, u_{q1}^{h_{n}}, \ldots, u_{qd_{q}}^{h_{n}}\right)$, then the corresponding error estimators $\{\eta_{h_{n}}(U^{h_{n}}, T)\}_{T\in\mathscr{T}_{h_{n}}}$ are computable, while if we set $W=\mathcal{E}_{h_{n}}U:=\left(E_{h_{n}}^{(1)}u_{11}, \ldots, E_{h_{n}}^{(1)}u_{1d_{1}},\ldots, E_{h_{n}}^{(q)}u_{q1}, \ldots,  E_{h_{n}}^{(q)}u_{qd_{q}}\right)$, the corresponding error estimators $\{\eta_{h_{n}}(\mathcal{E}_{h_{n}}U, T)\}_{T\in\mathscr{T}_{h_{n}}}$ are uncomputable. In our following analysis, both these two error estimators for the FE approximations will be used. We call the uncomputable error estimators $\{\eta_{h_{n}}(\mathcal{E}_{h_{n}}U, T)\}_{T\in\mathscr{T}_{h_{n}}}$ the theoretical error estimators. Note that the error estimators are the same as those introduced in \cite{dai2015convergence}.
	
	We then start our analysis with investigating  relationships among  computable and theoretical error estimators $\{\tilde{\eta}_{h_{n}}(U^{(n)}, T)\}_{T\in\mathscr{T}_{h_{n}}}$ and $\{\tilde{\eta}_{h_{n}}(\tilde{\mathcal{E}}_{n}U, T)\}_{T\in\mathscr{T}_{h_{n}}}$ for the ParO approximations, together with $\{\eta_{h_{n}}(U^{h_{n}}, T)\}_{T\in\mathscr{T}_{h_{n}}}$ and  $\{\eta_{h_{n}}(\mathcal{E}_{h_{n}}U, T)\}_{T\in\mathscr{T}_{h_{n}}}$ for the FE approximations.
	
	The equivalence of computable and theoretical error estimators for the FE approximations is shown in \cite{ dai2015convergence}. The more explicit constants involved are given in \cite{doi:10.1137/15M1036877}.
	The equivalence of error estimators for the ParO approximations follows from the similar arguments when the approximations are produced by solving \cref{eq:pro_ei_pro} in the $n$-th refinement of \Cref{algo:fram_matr}. We summarize the above conclusions in the following lemma.
	\begin{lemma}\label{prop:eqqqqqq}
		Let $U=\left(u_{11}, \ldots,  u_{qd_{q}}\right)$, where $\{u_{ij}\}_{j=1}^{d_{i}}$ is any orthonormal basis of $M(\lambda_{i}) (i=1,\ldots,q)$. If \begin{align}
			\max_{(1,1)\leqslant(i,j)\leqslant(q,d_{q})}\Vert u_{ij}-\tilde{E}_{n}^{(i)}u_{ij}\Vert_{b}&\leqslant\sqrt{1+(2N)^{-1}}-1,\label{assum1}\\\max_{(1,1)\leqslant(i,j)\leqslant(q,d_{q})}\Vert u_{ij}-E_{h_{n}}^{(i)}u_{ij}\Vert_{b}&\leqslant\sqrt{1+(2N)^{-1}}-1\label{assum2},
		\end{align}then 
		\begin{align*}
			\tilde{\eta}_{h_{n}}^{2}(\tilde{\mathcal{E}}_{n}U,\Omega)\cong \tilde{\eta}_{h_{n}}^{2}(U^{(n)},\Omega),\quad \eta_{h_{n}}^{2}(\mathcal{E}_{h_{n}}U,\Omega)\cong \eta_{h_{n}}^{2}(U^{h_{n}},\Omega).
		\end{align*}
	\end{lemma}   
	
	The following conclusion follows from \Cref{prop:eqqqqqq} and tells that the D\"orfler property \cref{ine:dor} holds for theoretical error estimators if the computable error estimators satisfy the D\"orfler property \cref{ine:dor}.
	\begin{corollary}
		Given $\theta_{1}, \theta_{2}\in(0,1)$, if 
		\begin{align*}
			\sum_{T \in \hat{\mathscr{T}}_{h_{n}}}\tilde{\eta}_{h_{n}}^{2}(U^{(n)},T) \geqslant \theta_{1} \tilde{\eta}_{h_{n}}^{2}(U^{(n)},\Omega),\quad
			\sum_{T \in \hat{\mathscr{T}}_{h_{n}}}\eta_{h_{n}}^{2}(U^{h_{n}},T) \geqslant \theta_{2} \eta_{h_{n}}^{2}(U^{h_{n}},\Omega), 
		\end{align*}
		then there exists $\theta_{1}', \theta_{2}'\in(0,1)$ such that, for $U=\left(u_{11}, \ldots, u_{1d_1},\ldots,u_{q1},\ldots ,u_{qd_{q}}\right)$, where $\{u_{ij}\}_{j=1}^{d_{i}}$ is any orthonormal basis of $M(\lambda_{i}) (i=1,\ldots,q)$ satisfying \Cref{assum1} and \Cref{assum2}, there hold
		\begin{align*}
			\sum_{T \in \hat{\mathscr{T}}_{h_{n}}}\tilde{\eta}_{h_{n}}^{2}(\tilde{\mathcal{E}}_{n}U,T) \geqslant \theta_{1}' \tilde{\eta}_{h_{n}}^{2}(\tilde{\mathcal{E}}_{n}U,\Omega),\quad \sum_{T \in \hat{\mathscr{T}}_{h_{n}}}\eta_{h_{n}}^{2}(\mathcal{E}_{h_{n}}U,T) \geqslant \theta_{2}' \eta_{h_{n}}^{2}(\mathcal{E}_{h_{n}}U,\Omega).
		\end{align*}
	\end{corollary}
	
	The following proposition provides an estimation between theoretical error estimators for the ParO approximations and the FE approximations, which will play a crucial role in our analysis. The proof is provided in \Cref{proof:prop:error_est}.
	\begin{proposition}\label{prop:error_est}
		If $h_{0}\ll1$ and $U=\left(u_{11}, \ldots,  u_{qd_{q}}\right)$, where $\{u_{ij}\}_{j=1}^{d_{i}}$ is an orthonormal basis of $M(\lambda_{i}) (i=1,\ldots,q)$ with $b(u_{ij}, u_{kl})=\delta_{ik}\delta_{jl}$, then 
		\begin{align*}
			\left|\eta_{h_{n}}(\mathcal{E}_{h_{n}}U, \Omega)-\tilde{\eta}_{h_{n}}(\tilde{\mathcal{E}}_{n}U, \Omega)\right|^{2}\leqslant C\sum_{i=1}^{q}\left(\operatorname{dist}^{2}_{a}\left(M_{h_{n}}(\lambda_{i}), \tilde{M}_{n}(\lambda_{i})\right)+\sum_{j=1}^{d_{i}}\left|\lambda_{ij}^{h_{n}}-\lambda_{ij}^{(n)}\right|^{2}\right),
		\end{align*}
		where $C>0 $ is a constant being independent of the mesh size.
	\end{proposition}
	\begin{remark}
		\Cref{prop:error_est} can be further strengthened. Taking \Cref{algo:pou_n_shifted} as an example, if  $\{\lambda_{ij}^{(n)}\}$ is produced by solving the projected eigenvalue problem \cref{eq:pro_ei_pro} and $\operatorname{dist}_{a}\left(M_{h_{n}}(\lambda_{i}), \tilde{M}_{n}(\lambda_{i})\right)\ll1$, it follows from \Cref{thm:cluster_eigen}  that
		\begin{align*}
			\left|\lambda_{ij}^{h_{n}}-\lambda_{ij}^{(n)}\right|\lesssim \operatorname{dist}^{2}_{a}\left(M_{h_{n}}(\lambda_{i}), \tilde{M}_{n}(\lambda_{i})\right),\quad (1,1)\leqslant(i,j)\leqslant(q,d_{q}),
		\end{align*}
        which leads to
		\begin{align*}
			\left|\eta_{h_{n}}(\mathcal{E}_{h_{n}}U, \Omega)-\tilde{\eta}_{h_{n}}(\tilde{\mathcal{E}}_{n}U, \Omega)\right|^{2}\lesssim \sum_{i=1}^{q}\operatorname{dist}^{2}_{a}\left(M_{h_{n}}(\lambda_{i}), \tilde{M}_{n}(\lambda_{i})\right).
		\end{align*}
	\end{remark}
	
	The following proposition illustrates that if the ParO iterations reach a certain accuracy in the $n$-th refinement, and the D\"orfler property \cref{ine:dor} is satisfied for the error estimators of the ParO approximations, then the property holds for the error estimators for the FE approximations as well. The proof is provided in \Cref{proof:thm:dof}.
	\begin{proposition}\label{thm:dof}
		Given $\theta\in(0,1)$, suppose $h_{n}\leqslant h_{0}\ll1$ and there exists $\rho>0$ satisfying
		\begin{align}\label{ine:theta}
		    \sqrt{2\left(C\rho+2\sqrt{C\rho}+2\right)C\rho}<\theta,
		\end{align}
 where $C$ is the constant appearing in \Cref{prop:error_est}, such that
		\begin{align}\label{ine:approxxxx}
			\sum_{i=1}^{q}\left(\operatorname{dist}^{2}_{a}\left(M_{h_{n}}(\lambda_{i}), \tilde{M}_{n}(\lambda_{i})\right)+\sum_{j=1}^{d_{i}}\left|\lambda_{ij}^{h_{n}}-\lambda_{ij}^{(n)}\right|^{2}\right)\leqslant\rho\tilde{\eta}_{h_{n}}^{2}\left(\tilde{\mathcal{E}}_{n}U, \Omega\right).
		\end{align}
		If
		\begin{align}\label{ine:dor_eta}
			\sum_{T \in \hat{\mathscr{T}}_{h_{n}}}\tilde{\eta}_{h_{n}}^{2}(\tilde{\mathcal{E}}_{n}U,T) \geqslant \theta \tilde{\eta}_{h_{n}}^{2}(\tilde{\mathcal{E}}_{n}U,\Omega),
		\end{align}
		then 
		\begin{align}\label{ine:dorrr}
			\sum_{T \in \hat{\mathscr{T}}_{h_{n}}}\eta_{h_{n}}^{2}(\mathcal{E}_{h_{n}}U,T) \geqslant \tilde{\theta} \eta_{h_{n}}^{2}(\mathcal{E}_{h_{n}}U,\Omega)\quad \text{with}\quad \tilde{\theta}\in\left(0, \frac{\theta-\sqrt{2\left(C\rho+2\sqrt{C\rho}+2\right)C\rho}}{\left(1+\sqrt{C\rho}\right)^{2}}\right].
		\end{align}
	\end{proposition}
	
	\subsection{Contraction property}
	\Cref{thm:dof} shows that if the ParO iterations reach a certain accuracy, then the fact that the error estimators for the ParO approximations satisfy the D\"orfler property implies that the estimators for the FE approximations satisfy the property as well. As a result, the error estimate for the adaptive ParO approximations can be transformed into the error estimate for the adaptive FE approximations. In our analysis, the following contraction property of adaptive FE approximations produced by \Cref{algo:pou_n_shifted} will be used (see Theorem 4.2 of \cite{dai2015convergence}).
	\begin{lemma}\label{prop:discre_con}
		 If $h_{0}\ll1$ and there exists $\hat{\theta}\in(0,1)$ such that
		 \begin{align*}
			\sum_{T \in \hat{\mathscr{T}}_{h_{n}}}\eta_{h_{n}}^{2}(\mathcal{E}_{h_{n}}U,T) \geqslant \hat{\theta} \eta_{h_{n}}^{2}(\mathcal{E}_{h_{n}}U,\Omega),\quad \forall n\geqslant0,
		\end{align*}
 then there exist constants $\gamma>0$ and $\alpha\in(0,1)$ that is independent of the mesh size $h_{n} (n\geqslant0)$ such that  
		\begin{align*}
			\left\Vert U-\mathcal{E}_{h_{n+1}}U\right\Vert_{a}^{2}+\gamma\eta_{h_{n+1}}^{2}(\mathcal{E}_{h_{n+1}}U,\Omega)\leqslant\alpha^{2}\left( \left\Vert U-\mathcal{E}_{h_{n}}U\right\Vert_{a}^{2}+\gamma\eta_{h_{n}}^{2}(\mathcal{E}_{h_{n}}U,\Omega)\right).
		\end{align*}
	\end{lemma}
	
	The following proposition tells the contraction property of the adaptive ParO approximations produced by \Cref{algo:pou_n_shifted}. 
	\begin{proposition}\label{thm:iter_con}
		Suppose $h_{0}\ll1$ and $\rho\ll1$ such that for $n\geqslant0$, 
		\begin{align}
			\sum_{i=1}^{q}\left(\operatorname{dist}^{2}_{a}\left(M_{h_{n}}(\lambda_{i}), \tilde{M}_{n}(\lambda_{i})\right)+\sum_{j=1}^{d_{i}}\left|\lambda_{ij}^{h_{n}}-\lambda_{ij}^{(n)}\right|^{2}\right)\leqslant&\rho\tilde{\eta}_{h_{n}}^{2}\left(\tilde{\mathcal{E}}_{n}U, \Omega\right)\label{ine:assu1},\\
			\sum_{i=1}^{q}\left(\operatorname{dist}^{2}_{a}\left(M_{h_{n+1}}(\lambda_{i}), \tilde{M}_{n+1}(\lambda_{i})\right)+\sum_{j=1}^{d_{i}}\left|\lambda_{ij}^{h_{n+1}}-\lambda_{ij}^{(n+1)}\right|^{2}\right)\leqslant&\rho\tilde{\eta}_{h_{n}}^{2}\left(\tilde{\mathcal{E}}_{n}U, \Omega\right)\label{ine:assu2}.
		\end{align}
		Then there exists the constant $\beta\in(0,1)$ that is independent of the mesh size $h_{n} (n\geqslant0)$ such that for any $n\geqslant0$,
		\begin{align*}
			\sum_{i=1}^{q}\left(\operatorname{dist}_{a}^{2}\left(M(\lambda_{i}), \tilde{M}_{n}(\lambda_{i})\right)+\sum_{j=1}^{d_{i}}\left|\lambda_{ij}^{(n)}-\lambda_{i}\right|\right)\lesssim \beta^{2n}.
		\end{align*}
	\end{proposition}
	
	\begin{proof}
	Since $\tilde{E}_{n}^{(i)}:=\mathcal{P}^{(i)}\circ E_{h_{n}}^{(i)}$ and $\mathcal{P}^{(i)}$ is an orthogonal projection with respect to $a(\cdot, \cdot)$ and $h_{n}\leqslant h_{0}\ll1$, we see that
	\begin{equation}
	   \begin{aligned}\label{ine:maxnum}
			&\left\Vert\mathcal{E}_{h_{n+1}}U-\tilde{\mathcal{E}}_{n+1}U\right\Vert_{a}^{2}\\\leqslant&\sum_{i=1}^{q}\sum_{j=1}^{d_{i}}\left\Vert E_{h_{n+1}}^{(i)}u_{ij}-\tilde{E}_{n+1}^{(i)}u_{ij} \right\Vert_{a}^{2}=\sum_{i=1}^{q}\sum_{j=1}^{d_{i}}\left\Vert \left(\mathrm{I}-\mathcal{P}^{(i)}\right)E_{h_{n+1}}^{(i)}u_{ij}\right\Vert_{a}^{2}\\\leqslant&\sum_{i=1}^{q}\sum_{j=1}^{d_{i}}\operatorname{dist}^{2}_{a}\left(E_{h_{n+1}}^{(i)}u_{ij}, \tilde{M}_{n+1}(\lambda_{i})\right)\leqslant D\sum_{i=1}^{q}\operatorname{dist}^{2}_{a}\left(M_{h_{n+1}}(\lambda_{i}), \tilde{M}_{n+1}(\lambda_{i})\right),
		\end{aligned} 
	\end{equation}
		where $D:=\max_{1\leqslant i\leqslant q}d_{i}$ is the maximum  multiplicity of the eigenvalues we look for.
	
	We see that \cref{ine:theta} holds due to $\rho\ll1$. Then by Young's inequality, \Cref{prop:error_est}, \Cref{prop:discre_con}, \cref{ine:assu2} and \cref{ine:maxnum}, we have that for any $\delta>0$, 
		\begin{align*}
			&\left\Vert U-\tilde{\mathcal{E}}_{n+1}U\right\Vert_{a}^{2}+\gamma\tilde{\eta}_{h_{n+1}}^{2}(\tilde{\mathcal{E}}_{n+1}U,\Omega)\leqslant(1+\delta)\left( \left\Vert U-\mathcal{E}_{h_{n+1}}U\right\Vert_{a}^{2}+\gamma\eta_{h_{n+1}}^{2}(\mathcal{E}_{h_{n+1}}U,\Omega)\right)\\&+\left(1+\frac{1}{\delta}\right)\left(\gamma\left|\tilde{\eta}_{h_{n+1}}(\tilde{\mathcal{E}}_{n+1}U,\Omega)-\eta_{h_{n+1}}(\mathcal{E}_{h_{n+1}}U,\Omega)\right|^{2}+\left\Vert\mathcal{E}_{h_{n+1}}U-\tilde{\mathcal{E}}_{n+1}U\right\Vert_{a}^{2}\right)\\\leqslant&(1+\delta)\left( \left\Vert U-\mathcal{E}_{h_{n+1}}U\right\Vert_{a}^{2}+\gamma\eta_{h_{n+1}}^{2}(\mathcal{E}_{h_{n+1}}U,\Omega)\right)\\+&\left(1+\frac{1}{\delta}\right)\left(\gamma C\sum_{i=1}^{q}\sum_{j=1}^{d_{i}}\left|\lambda_{ij}^{h_{n+1}}-\lambda_{ij}^{(n+1)}\right|^{2}+\left(D+\gamma C\right)\sum_{i=1}^{q}\operatorname{dist}^{2}_{a}\left(M_{h_{n+1}}(\lambda_{i}), \tilde{M}_{n+1}(\lambda_{i})\right)\right)\\\leqslant&(1+\delta)\alpha^{2}\left( \left\Vert U-\mathcal{E}_{h_{n}}U\right\Vert_{a}^{2}+\gamma\eta_{h_{n}}^{2}(\mathcal{E}_{h_{n}}U,\Omega)\right)+\left(1+\frac{1}{\delta}\right)\left(D+\gamma C\right)\rho\tilde{\eta}_{h_{n}}^{2}\left(\tilde{\mathcal{E}}_{n}U, \Omega\right).
		\end{align*}
		In further, by applying the similar analysis to $\left\Vert U-\mathcal{E}_{h_{n}}U\right\Vert_{a}^{2}+\gamma\eta_{h_{n}}^{2}(\mathcal{E}_{h_{n}}U,\Omega)$, we have
		\begin{align*}
			&\left\Vert U-\tilde{\mathcal{E}}_{n+1}U\right\Vert_{a}^{2}+\gamma\tilde{\eta}_{h_{n+1}}^{2}(\tilde{\mathcal{E}}_{n+1}U,\Omega)\leqslant(1+\delta)^{2}\alpha^{2}\left( \left\Vert U-\tilde{\mathcal{E}}_{n}U\right\Vert_{a}^{2}+\gamma\tilde{\eta}_{h_{n}}^{2}(\tilde{\mathcal{E}}_{n}U,\Omega)\right)\\&+(1+\delta)\left(1+\frac{1}{\delta}\right)\alpha^{2}\left(\gamma\left|\tilde{\eta}_{h_{n}}(\tilde{\mathcal{E}}_{n}U,\Omega)-\eta_{h_{n}}(\mathcal{E}_{h_{n}}U,\Omega)\right|^{2}+\left\Vert\mathcal{E}_{h_{n}}U-\tilde{\mathcal{E}}_{n}U\right\Vert_{a}^{2}\right)\\&+
			\left(1+\frac{1}{\delta}\right)\left(D+\gamma C\right)\rho\tilde{\eta}_{h_{n}}^{2}\left(\tilde{\mathcal{E}}_{n}U, \Omega\right)\\\leqslant&(1+\delta)^{2}\alpha^{2}\left( \left\Vert U-\tilde{\mathcal{E}}_{n}U\right\Vert_{a}^{2}+\gamma\tilde{\eta}_{h_{n}}^{2}(\tilde{\mathcal{E}}_{n}U,\Omega)\right)+\left(1+\frac{1}{\delta}\right)\left(D+\gamma C\right)\rho\tilde{\eta}_{h_{n}}^{2}\left(\tilde{\mathcal{E}}_{n}U, \Omega\right)\\&+(1+\delta)\left(1+\frac{1}{\delta}\right)\alpha^{2}\left(\gamma C\sum_{i=1}^{q}\sum_{j=1}^{d_{i}}\left|\lambda_{ij}^{h_{n}}-\lambda_{ij}^{(n)}\right|^{2}+\left(D+\gamma C\right) \sum_{i=1}^{q}\operatorname{dist}^{2}_{a}\left(M_{h_{n}}(\lambda_{i}), \tilde{M}_{n}(\lambda_{i})\right)\right).
		\end{align*}
		Consequently, we obtain from \cref{ine:assu1} that 
		\begin{align*}
			&\left\Vert U-\tilde{\mathcal{E}}_{n+1}U\right\Vert_{a}^{2}+\gamma\tilde{\eta}_{h_{n+1}}^{2}(\tilde{\mathcal{E}}_{n+1}U,\Omega)\\\leqslant&(1+\delta)^{2}\alpha^{2} \left\Vert U-\tilde{\mathcal{E}}_{n}U\right\Vert_{a}^{2}+\left((1+\delta)^{2}\alpha^{2}\gamma+\left(1+(1+\delta)\alpha^{2}\right)\left(1+\frac{1}{\delta}\right)(D+\gamma C)\rho\right)\tilde{\eta}_{h_{n}}^{2}(\tilde{\mathcal{E}}_{n}U,\Omega).
		\end{align*}
		Since $\rho\ll1$ and $\alpha \in (0, 1)$, we can choose $\delta>0$ satisfying
		\begin{align*}
			(1+\delta)^{2}\alpha^{2}\gamma+\left(1+(1+\delta)\right)\alpha^{2}\left(D+\frac{1}{\delta}\right)(1+\gamma C)\rho<\gamma.
		\end{align*}
		Then
		\begin{align*}
			\beta:=\sqrt{(1+\delta)^{2}\alpha^{2}+\frac{\left(1+(1+\delta)\right)\alpha^{2}\left(D+1/\delta\right)(1+\gamma C)\rho}{\gamma}}\in (0,1).
		\end{align*}
		Therefore, there holds for any two consecutive iterations $n$ and $n+1$ that  
		\begin{align*}
			\left\Vert U-\tilde{\mathcal{E}}_{n+1}U\right\Vert_{a}^{2}+\gamma\tilde{\eta}_{h_{n+1}}^{2}(\tilde{\mathcal{E}}_{n+1}U,\Omega)\leqslant\beta^{2}\left( \left\Vert U-\tilde{\mathcal{E}}_{n}U\right\Vert_{a}^{2}+\gamma\tilde{\eta}_{h_{n}}^{2}(\tilde{\mathcal{E}}_{n}U,\Omega)\right).
		\end{align*}
		Note that $\beta\in(0,1)$ is independent of the mesh size $h_{n} (n\geqslant0)$.  
		
		For any $v_{i}\in M(\lambda_{i})$ and  $\{u_{ij}\}_{j=1}^{d_{i}}$ as an orthonormal basis of $M(\lambda_{i})$, 
		we obtain from the fact that 
		\begin{align*}
		    a(u_{ik}, u_{il})=\lambda_{i}b(u_{ik}, u_{il})=0,\quad k\neq l
		\end{align*}
		and \Cref{lem:dissyysybs} that
		\begin{align*}
			\sum_{i=1}^{q} \operatorname{dist}_{a}^{2}\left(v_{i}, \tilde{M}_{n}(\lambda_{i})\right)\leqslant&\sum_{i=1}^{q}\sum_{j=1}^{d_{i}} \operatorname{dist}_{a}^{2}\left(u_{ij}, \tilde{M}_{n}(\lambda_{i})\right)\leqslant\sum_{i=1}^{q}\sum_{j=1}^{d_{i}}\left\Vert u_{ij}-\tilde{E}_{n}^{(i)}u_{ij} \right\Vert_{a}^{2}\\\lesssim&\left\Vert U-\tilde{\mathcal{E}}_{n+1}U\right\Vert_{a}^{2}+\gamma\tilde{\eta}_{h_{n+1}}^{2}(\tilde{\mathcal{E}}_{n+1}U,\Omega)\lesssim\beta^{2n}.
		\end{align*}
		Due to the arbitrariness of $\{v_{i}\}_{i=1}^q$, we have 
		\begin{align*}
			\sum_{i=1}^{q} \operatorname{dist}_{a}^{2}\left(M(\lambda_{i}), \tilde{M}_{n}(\lambda_{i})\right)\lesssim\beta^{2n}.
		\end{align*}
		
		For $\psi\in \bigoplus_{i=1}^{q}M(\lambda_{i})$, by \Cref{lem:dissyysybs} again, there holds
		\begin{equation*}
			\begin{aligned}
				&\operatorname{dist}_{a}\left(\psi, \bigoplus_{i=1}^{q}\tilde{M}_{n}(\lambda_{i})\right)\leqslant\sqrt{\sum_{i=1}^{q}\sum_{j=1}^{d_{i}}\operatorname{dist}_{a}^{2}\left(u_{ij}, \bigoplus_{i=1}^{q}\tilde{M}_{n}(\lambda_{i})\right)}\\\leqslant&\sqrt{\sum_{i=1}^{q}\sum_{j=1}^{d_{i}}\operatorname{dist}_{a}^{2}\left(M(\lambda_{i}), \bigoplus_{i=1}^{q}\tilde{M}_{n}(\lambda_{i})\right)}\leqslant\sqrt{\sum_{i=1}^{q}\sum_{j=1}^{d_{i}}\operatorname{dist}_{a}^{2}\left(M(\lambda_{i}), \tilde{M}_{n}(\lambda_{i})\right)}\lesssim\beta^{n},
			\end{aligned} 
		\end{equation*}
		which implies that 
		\begin{align}\label{ine:piece_to_cluster}
			\operatorname{dist}_{a}\left(\bigoplus_{i=1}^{q}M(\lambda_{i}), \bigoplus_{i=1}^{q}\tilde{M}_{n}(\lambda_{i})\right)\lesssim\beta^{n}.
		\end{align}
		Hence, it follows from $h_{0}\ll1$, \cref{ine:piece_to_cluster} and \Cref{thm:cluster_eigen} applied to \cref{eq:pro_ei_pro} that 
		\begin{align*}
			\left|\lambda_{ij}^{(n)}-\lambda_{ij}\right|\lesssim\beta^{2n},\quad \forall (1,1)\leqslant(i,j)\leqslant(q,d_{q}), 
		\end{align*}
		which completes the proof.
	\end{proof}
	
	\Cref{prop:eqqqqqq} and \Cref{prop:error_est} tell the relationships among different error estimators, that is,
	\begin{align*}
	    \tilde{\eta}_{h_{n}}^{2}\left(U^{(n)}, \Omega\right)\cong\tilde{\eta}_{h_{n}}^{2}\left(\tilde{\mathcal{E}}_{n}U, \Omega\right)\approx\eta_{h_{n}}^{2}\left(\mathcal{E}_{h_{n}}U, \Omega\right)\cong\eta_{h_{n}}^{2}\left(U^{h_{n}}, \Omega\right).
	\end{align*}
	We see that to satisfy the conditions \cref{ine:assu1} and \cref{ine:assu2} and to ensure
	the convergence of \Cref{algo:pou_n_shifted}, we only need to match the errors of approximations produced by the ParO iterations to the discretization accuracy. When the discretization is not accurate enough, we do not need to do the ParO iterations too many times. We mention in \cite{dai2025numerical} that during the implementation process of the ParO iterations, we are able to obtain quasi-orthogonal approximations for eigenfunctions. The ParO iterations utilizes the quasi-orthogonal approximations, for which the computational cost is lower, to obtain orthogonal approximations. In the following subsections, we will show that the errors of the approximations produced by the ParO iterations can be matched to the FE discretization accuracy by the quasi-orthogonality.
	
	\subsection{Quasi-orthogonality}
	To analyze the convergence of \Cref{algo:pou_n_shifted}, we recall the two existing results, which involve the quasi-orthogonality \cite{dai2025numerical}.
	
	The following lemma tells the approximation property of orthonormal bases by the distance from one subspace to another (see Proposition 3.4 in \cite{dai2025numerical}). 
	\begin{lemma}\label{prop:subspace_angle}
		Given two subspaces $X, Y\subset H_{0}^{1}(\Omega)$ with $ \operatorname{dim}(X)=\operatorname{dim}(Y)=n$. If $ \operatorname{dist}_{a}(X, Y)<1,$
		then for any orthonormal basis $\{x_{j}\}_{j=1}^{n}$ of $X$ with $a(x_{i}, x_{j})=\delta_{ij}$, there exists an orthonormal basis $\{y_{j}\}_{j=1}^{n}$ of $Y$ with $a(y_{i}, y_{j})=\delta_{ij}$ satisfying
		\begin{align*}
			\operatorname{dist}_{a}(x_{j}, y_{j})\leqslant(1+\sqrt{n})\sqrt{2-2\sqrt{1-\operatorname{dist}^{2}_{a}(X, Y)}},\quad j=1,2,\ldots,n.
		\end{align*}
	\end{lemma} 
	
	In the $n$-th refinement, we set 
	\begin{align*}
		g_{n}:=\min_{1\leqslant i\leqslant q}\left|\lambda^{h_{n}}_{id_{i}}-\lambda^{h_{n}}_{(i+1)1}\right|,\quad \gamma_{n}:=\max_{1\leqslant i\leqslant q}\left(\lambda_{id_{i}}^{h_{n}}-\lambda_{i1}^{h_{n}}\right). 
	\end{align*}
	We obtain from the minimum-maximum principle that $g_{n}$ increases towards
	\begin{align*}
		\min_{1\leqslant i\leqslant q}(\lambda_{i+1}-\lambda_{i}),
	\end{align*}
	and $\gamma_{n}$ decreases towards $0$ as $n\rightarrow\infty$ and $h_{n}\rightarrow0$. 
	
	The following lemma provides the error estimate of \Cref{algo:pou_n_shifted} in the $n$-th refinement (see Theorem 4.4 in \cite{dai2025numerical} for more details). 
	\begin{lemma}\label{prop:clus_pou_n_shifted}Assume that $h_{n}\ll1$, and 
		there exists $0<\varepsilon_{n,0}\ll1$ and an orthonormal basis $\{u^{h_{n},0}_{ij}\}_{j=1}^{d_{i}}$ of $M_{h_{n}}(\lambda_{i}) (i=1,\ldots,q)$ with $b(u^{h_{n},0}_{ij}, u^{h_{n},0}_{kl})=\delta_{ik}\delta_{jl}$ such that for $(1,1)\leqslant(i,j)\leqslant(q,d_{q})$,
		\begin{align}\label{ine:iniiini}
			\operatorname{dist}_{a}\left(u^{h_{n},0}_{ij}, u_{i j}^{(n,0)}\right)\leqslant\varepsilon_{n,0},\quad 
			\left|\lambda^{h_{n}}_{ij}-\lambda_{ij}^{(n,0)}\right|\ll g_{n}.
		\end{align}
		If $\{u_{ij}^{(n,m)}\}_{i,j}$ are produced by \Cref{algo:pou_n_shifted}, then there exists an orthonormal basis $\{u^{h_{n},m}_{ij}\}_{j=1}^{d_{i}}$ of $M_{h_{n}}(\lambda_{i}) (i=1,\ldots,q)$ with $b(u^{h_{n},m}_{ij}, u^{h_{n},m}_{kl})=\delta_{ik}\delta_{jl}$, and a sequence $\{\varepsilon_{n,m}\}_{m\geqslant1}$ such that for $(1,1)\leqslant(i,j)\leqslant(q,d_{q})$,
		\begin{align*}
			\operatorname{dist}_{a}\left(u^{h_{n},m}_{ij}, u_{ij}^{(n,m)}\right)\leqslant\varepsilon_{n,m},\quad \left|\lambda^{h_{n}}_{ij}-\lambda_{ij}^{(n,m)}\right|\leqslant\frac{\lambda^{h_{n}}_{q+1, 1}}{\tilde{C}_{**}^{2}}\varepsilon_{n,m}^{2},\quad 
		\end{align*}
		where $\tilde{C}_{**}$ is a constant that is independent of $\tilde{X}_{n,m} (m=1,2,\ldots)$ and $\varepsilon_{n,m}\searrow 0$ as $m\rightarrow\infty$.
	\end{lemma}
	
	\begin{remark}\label{rem:obser}
		We observe from the proof of \Cref{prop:clus_pou_n_shifted} (Theorem 4.4 in \cite{dai2025numerical}) that we can choose the sequence $\{\varepsilon_{n,0}\}_{n\geqslant1}$ that increases with respect to $n$ since $g_{n}$ increases and $\gamma_{n}$ decreases as the mesh is refined, and $\tilde{C}_{**}$ can be chosen to be independent of the mesh size when $h_{0}\ll1$. Moreover, it has been shown in \cite{dai2025numerical} that
		\begin{align*}
			\lim_{n\rightarrow\infty}\lim_{m\rightarrow\infty}\frac{\varepsilon_{n,m+1}}{\varepsilon_{n,m}}=0.
		\end{align*}  That is, the finer the mesh, the better the ParO iteration approximates.
		
	\end{remark}
	
	The following proposition shows that after some ParO iterations, the error between the ParO approximations and the FE approximations can be controlled by the error estimator for the FE discretization in each refinement.
	
	\begin{proposition}\label{prop:paro_conver_given}
		Let $\left\{\left(\lambda_{ij}^{(n)}, u_{ij}^{(n)}\right)\right\}$ be produced by
		\Cref{algo:pou_n_shifted} with $m_{n}$ steps of ParO iterations in the $n$-th refinement, and $\tilde{M}_{n}(\lambda_{i})=\operatorname{span}\{u_{i1}^{(n)}, \ldots, u_{id_{i}}^{(n)}\}$ for $i=1,\ldots,q$. For the initial data $\left\{\left(\lambda_{ij}^{(-1)}, u_{ij}^{(-1)}\right)\right\}$,  there exists an orthonormal basis $\{u^{h_{0}, 0}_{ij}\}_{j=1}^{d_{i}}$ of $M_{h_{0}}$ with $b(u^{h_{0}, 0}_{ij}, u^{h_{0}, 0}_{kl})=\delta_{ik}\delta_{jl}$ such that 
		\begin{align}\label{ine:initial_condition}
			\operatorname{dist}_{a}\left(u^{h_{0}, 0}_{ij},u_{ij}^{(-1)}\right)\ll1,\quad  \left|\lambda^{h_{0}}_{ij}-\lambda_{ij}^{(-1)}\right|\ll g_{0},\quad (1,1)\leqslant(i,j)\leqslant(q,d_{q}).
		\end{align}
		If $h_{0}\ll1$, then for any $\tilde{\rho}\ll1$ and any $n\geqslant0$, there exists $m_{n}$ large enough such that 
		\begin{equation}
			\begin{aligned}\label{ine:initial_con}
				&\sum_{i=1}^{q}\left(\operatorname{dist}^{2}_{a}\left(M_{h_{n}}(\lambda_{i}), \tilde{M}_{n}(\lambda_{i})\right)+\sum_{j=1}^{d_{i}}\left|\lambda_{ij}^{h_{n}}-\lambda_{ij}^{(n)}\right|^{2}\right)\\\leqslant&\tilde{\rho}\min\{\tilde{\eta}_{h_{n-1}}^{2}(\tilde{\mathcal{E}}_{n-1}U, \Omega), \eta_{h_{n}}^{2}\left(\mathcal{E}_{h_{n}}U, \Omega\right)\},
			\end{aligned}
		\end{equation}
		where $\tilde{\eta}_{h_{-1}}^{2}(\tilde{\mathcal{E}}_{-1}U, \Omega):=\infty$.
		
	\end{proposition}
	\begin{proof}
		We start with $n=0$.
		
		Let $\{\varepsilon_{n,m}\}$ be the sequence appearing in \Cref{prop:clus_pou_n_shifted} and satisfies that $\{ \varepsilon_{n,0}\}_{n\geqslant 1}$  increases with respect to $n$ (see \Cref{rem:obser}). Then we obtain from $h_0\ll 1$ and the correlation between the mesh and $\{ \varepsilon_{n,0}\}_{n\geqslant 1}$ that
		\begin{align*}
			\operatorname{dist}_{a}(M(\lambda_{i}), V^{h_{0}})\ll\varepsilon_{0,0}.
		\end{align*}
		
		We see that $\{u_{ij}\}_{j=1}^{d_{i}}$ is an orthonormal basis of $M(\lambda_{i})$ and  $E_{h_{0}}^{(i)}$ is an orthogonal projection with respect to $a(\cdot, \cdot)$. We obtain from \Cref{lem:dissyysybs} and \Cref{thm:cluster_eigenfunc} that if $u\in M(\lambda_{i})$, then 
		\begin{equation}\label{ine:dist_vec_spa}
			\begin{aligned}
				\operatorname{dist}_{a}(u, M_{h_{0}}(\lambda_{i}))\leqslant \sqrt{d_{i}}\max_{1\leqslant j\leqslant d_{i}}\operatorname{dist_{a}}(u_{ij}, E_{h_{0}}^{(i)}u_{ij})\lesssim\operatorname{dist}_{a}(M(\lambda_{i}), V^{h_{0}}), \quad i=1,\ldots,q,
			\end{aligned}  
		\end{equation}
		when $h_{0}\ll1$, that is
		\begin{align}\label{ine:h0}
			\operatorname{dist}_{a}(M(\lambda_{i}), M_{h_{0}}(\lambda_{i}))\ll\varepsilon_{0,0}, \quad i=1,\ldots,q.
		\end{align}
		
		By \cref{ine:initial_condition}, we have that 
		\begin{align*}
			\operatorname{dist}_{a}\left(u^{h_{0}, 0}_{ij},u_{ij}^{(0,0)}\right)\leqslant\varepsilon_{0,0},\quad  \left|\lambda^{h_{0}}_{ij}-\lambda_{ij}^{(0,0)}\right|\ll g_{0},\quad (1,1)\leqslant(i,j)\leqslant(q,d_{q}).
		\end{align*}
		
		We then derive from \Cref{prop:clus_pou_n_shifted} that there exists $m_{0}$ large enough, such that for the ParO approximations $\left\{\left(\lambda_{ij}^{(0)}, u_{ij}^{(0)}\right)\right\}=\left\{\left(\lambda_{ij}^{(0, m_{0})}, u_{ij}^{(0, m_{0})}\right)\right\}$, there exists an orthonormal basis $\{u_{ij}^{h_{0}, m_{0}}\}_{j=1}^{d_{i}}$ of $M_{h_{0}}(\lambda_{i})$ with $b\left(u_{ij}^{h_{0}, m_{0}}, u_{kl}^{h_{0}, m_{0}}\right)=\delta_{ik}\delta_{jl}$ satisfying for $(1,1)\leqslant(i,j)\leqslant(q,d_{q})$,
		\begin{equation}
			\begin{aligned}\label{ine:paro_res}
				&\operatorname{dist}_{a}(u_{ij}^{h_{0}, m_{0}}, u_{ij}^{(0)})\ll \min\{\tilde{\eta}_{h_{-1}}(\tilde{\mathcal{E}}_{-1}U, \Omega), \eta_{h_{0}}\left(\mathcal{E}_{h_{0}}U, \Omega\right),\varepsilon_{0,0}\},\\\quad &\left|\lambda_{ij}^{(0)}-\lambda_{ij}^{h_{0}}\right|\ll \min\{\tilde{\eta}_{h_{-1}}^{2}(\tilde{\mathcal{E}}_{-1}U, \Omega), \eta_{h_{0}}^{2}\left(\mathcal{E}_{h_{0}}U, \Omega\right),g_{0}\}.
			\end{aligned}
		\end{equation}
		
		Next, we show that $\left\{\left(\lambda_{ij}^{(1,0)}, u_{ij}^{(1,0)}\right)\right\}:=\left\{\left(\lambda_{ij}^{(0)}, u_{ij}^{(0)}\right)\right\}$ satisfies the initial condition \cref{ine:iniiini} in the next refinement.

		Since $\{u^{h_{0}, m_{0}}_{ij}\}_{j=1}^{d_{i}}$ is an orthonormal basis of $M_{h_{0}}(\lambda_{i})$, for $\varphi\in M_{h_{0}}(\lambda_{i})$, we obtain from \Cref{lem:dissyysybs} that
		\begin{align*}
			&\operatorname{dist}_{a}\left(\varphi, \tilde{M}_{0}(\lambda_{i})\right)\leqslant \sqrt{\sum_{j=1}^{d_{i}}\operatorname{dist}_{a}^{2}\left(u^{h_{0}, m_{0}}_{ij},\tilde{M}_{0}(\lambda_{i})\right)}\leqslant \sqrt{\sum_{j=1}^{d_{i}}\operatorname{dist}_{a}^{2}\left(u^{h_{0}, m_{0}}_{ij},u_{ij}^{(1,0)}\right)},
		\end{align*}
		which together with \cref{ine:paro_res} yields
		\begin{align}\label{ine:clu_pieces}
			\operatorname{dist}_{a}\left(M_{h_{0}}(\lambda_{i}), \tilde{M}_{0}(\lambda_{i})\right)\leqslant\sqrt{\sum_{j=1}^{d_{i}}\operatorname{dist}_{a}^{2}\left(u^{h_{0},m_{0}}_{ij},u_{ij}^{(1,0)}\right)}\ll\varepsilon_{0,0},\quad i=1,\ldots,q.
		\end{align}
		
		It follows from \cref{ine:h0}, \cref{ine:clu_pieces} and similar results for $M_{h_{1}}(\lambda_{i})$ that
		\begin{align*}
			\operatorname{dist}_{a}(M_{h_{1}}(\lambda_{i}), \tilde{M}_{0}(\lambda_{i}))\leqslant&\operatorname{dist}_{a}(M_{h_{1}}(\lambda_{i}), M(\lambda_{i}))+\operatorname{dist}_{a}(M(\lambda_{i}), M_{h_{0}}(\lambda_{i}))\\&+\operatorname{dist}_{a}(M_{h_{0}}(\lambda_{i}), \tilde{M}_{0}(\lambda_{i}))\ll\varepsilon_{0,0}.
		\end{align*}
		Hence we get from \Cref{prop:subspace_angle} that there exists an orthonormal basis $\{u_{ij}^{h_{1},0}\}_{j=1}^{d_{i}}$ of $M_{h_{1}}(\lambda_{i}) (i=1,\ldots,q)$ with $b\left(u_{ij}^{h_{1},0}, u_{kl}^{h_{1},0}\right)=\delta_{ik}\delta_{jl}$ such that 
		\begin{align*}
			\operatorname{dist}_{a}\left(u_{ij}^{h_{1},0}, u_{ij}^{(1,0)}\right)\leqslant\varepsilon_{0,0}\leqslant\varepsilon_{1,0}.
		\end{align*}
		
		We have from \Cref{thm:cluster_eigen}, \cref{ine:initial_condition} and $h_{1}\leqslant h_{0}\ll 1, g_{1}\geqslant g_{0}$ that
		\begin{align*}
			\left|\lambda_{ij}^{(0)}-\lambda_{ij}^{h_{1}}\right|\leqslant\left|\lambda_{ij}^{(0)}-\lambda_{ij}^{h_{0}}\right|+\left|\lambda_{ij}^{h_{0}}-\lambda_{i}\right|+\left|\lambda_{i}-\lambda_{ij}^{h_{1}}\right|\ll g_{1}.
		\end{align*}
		That is, the initial condition \cref{ine:iniiini} for ParO iterations holds in the next refinement.   Consequently, as iterate continuously, \cref{ine:initial_con} holds for any $n\geqslant0$. We complete the proof.
	\end{proof}
	
	\subsection{Convergence}
	By investigating the relationships among error estimators and the contraction property of the adaptive ParO approximations, and incorporating the quasi-orthogonality, we establish the convergence of the shifted-inverse based adaptive ParO algorithm (\Cref{algo:pou_n_shifted}) as follows.
	\begin{theorem}\label{thm:paro_conver}
		Under the same assumptions as in \Cref{prop:paro_conver_given}, there exists a constant $\beta\in(0,1)$ that is independent of the mesh size $h_{n} (n\geqslant0)$ such that for any $n\geqslant0$,
		\begin{align*}
			\sum_{i=1}^{q}\left(\operatorname{dist}_{a}^{2}\left(M(\lambda_{i}), \tilde{M}_{n}(\lambda_{i})\right)+\sum_{j=1}^{d_{i}}\left|\lambda_{ij}^{(n)}-\lambda_{i}\right|\right)\lesssim \beta^{2n}.
		\end{align*}
	\end{theorem}
	\begin{proof}
		We see from \Cref{prop:paro_conver_given} that when $h_{0}\ll1$ and some proper initial values are provided, after several ParO iterations, the error between the ParO approximations and the FE approximations can be controlled by the error estimator for the FE discretization in each refinement. That is, for any $n\geqslant0$ and  $\tilde{\rho}\ll1$, after some ParO iterations, there holds that
		\begin{align*}
			\sum_{i=1}^{q}\left(\operatorname{dist}^{2}_{a}\left(M_{h_{n}}(\lambda_{i}), \tilde{M}_{n}(\lambda_{i})\right)+\sum_{j=1}^{d_{i}}\left|\lambda_{ij}^{h_{n}}-\lambda_{ij}^{(n)}\right|^{2}\right)\leqslant&\tilde{\rho}\eta_{h_{n}}^{2}\left(\mathcal{E}_{h_{n}}U, \Omega\right),
		\end{align*}
		which together with \Cref{prop:error_est} yields that
		\begin{align*}
			&\sum_{i=1}^{q}\left(\operatorname{dist}^{2}_{a}\left(M_{h_{n}}(\lambda_{i}), \tilde{M}_{n}(\lambda_{i})\right)+\sum_{j=1}^{d_{i}}\left|\lambda_{ij}^{h_{n}}-\lambda_{ij}^{(n)}\right|^{2}\right)\leqslant\frac{2\tilde{\rho}}{1-2\tilde{\rho}C}\tilde{\eta}_{h_{n}}^{2}\left(\tilde{\mathcal{E}}_{n}U, \Omega\right).
		\end{align*}
		Namely, the condition \cref{ine:assu1} is satisfied. \cref{ine:assu2} follows from \Cref{prop:paro_conver_given} in the $(n+1)$-th refinement. Then by \Cref{thm:iter_con}, we complete the proof.
	\end{proof}
	
	\Cref{thm:paro_conver} provides the convergence of approximations produced by \Cref{algo:pou_n_shifted}. We see that when $h_{0}\ll1$ and some proper initial values are provided (i.e., the initial orbitals are quasi-orthogonal, which means that they are in the neighbourhood of an orthogonal basis of the subspace $M_{h_{0}}(\lambda_{i})$ spanned by the FE approximations), to ensure the
	convergence of \Cref{algo:pou_n_shifted}, we only need to match the errors of approximations produced by the ParO iterative computations to the approximation accuracy of FE discretizations to satisfy \cref{ine:initial_con}. When the discretization is not accurate enough, we do not need to perform the ParO iterations too many times. Moreover, in accordance with the analysis in \cite{dai2025numerical}, with larger $g_{n}$ and smaller $\gamma_{n}$, the ParO iteration approximates better, which tells that as the iterations progress and the mesh updates, the speed of convergence will be  increasingly accelerated.
	
	\section{Concluding remarks}  
	The adaptive ParO method avoids the direct solution of large scale eigenvalue problems by solving source problems intrinsically in parallel in adaptive FE spaces to produce quasi-orthogonal approximations of orbitals, and solving some small scale eigenvalue problems in the lowest-dimensional subspace generated from the approximations. The numerical experiments in \cite{dai2014parallel} show the efficiency of the adaptive ParO method and its great potential for large scale parallelization, especially when many highly accurate eigenpair approximations are required. 
	
	In this paper, we have carried out the numerical analysis of the adaptive ParO method for solving eigenvalue problems of partial differential operators, which provides a mathematical justification for the numerical experiments in \cite{dai2014parallel}. With the investigation on the relationship among various types of error estimators, the contraction property of the adaptive ParO approximations, and the quasi-orthogonality of the approximations, we have obtained the convergence of approximations produced by the adaptive ParO method.
	We conclude from our analysis that to ensure the convergence of the approach, we only need to match the errors of approximations produced by the ParO iterations to the accuracy of FE approximations. When the discretization is not accurate enough, we do not need to perform the ParO iterations too many times, which will significantly enhance the computational efficiency. It is our on-going work to carry out the numerical analysis of the ParO method incorporating into other computational adaptive methods, such as the adaptive planewave method \cite{AAMM-16-636}.
	
	\appendix
	\section{Detailed proofs}\label{appen}
	\subsection{Proof of \Cref{lem:dissyysybs}}\label{proof:lem:dissyysybs}
	\begin{proof}
	We observe that $\operatorname{dist}_{a}(x, Y)$ and $\operatorname{dist}_{a}(x_{i}, Y)$ are independent of the norms of $x$ or $x_{i}$ for $i=1,\ldots, d$. 
	Then, without loss of generality, we set $\Vert x\Vert_{a}=\Vert x_{i}\Vert_{a}=1$ for $i=1,\ldots, d$. Let $\{\alpha_{i}\}_{i=1}^{d}$ be the constants such that $x=\sum_{i=1}^{d}\alpha_{i}x_{i}$ with $\sum_{i=1}^{d}\alpha_{i}^{2}=1$, and $\mathcal{P}_{Y}$ be the orthogonal projection from $X$ onto $Y$ with respect to $a(\cdot, \cdot)$. We have
		\begin{align*}
			\operatorname{dist}_{a}(x, Y)=\left\Vert x-\mathcal{P}_{Y}x\right\Vert_{a}=&\left\Vert\sum_{i=1}^{d}\alpha_{i}\left(\operatorname{I}-\mathcal{P}_{Y}\right)x_{i}\right\Vert_{a}\\\leqslant&\sum_{i=1}^{d}|\alpha_{i}|\left\Vert\left(\operatorname{I}-\mathcal{P}_{Y}\right)x_{i}\right\Vert_{a}\leqslant\sqrt{\sum_{i=1}^{d}\operatorname{dist}_{a}^{2}(x_{i}, Y)},
		\end{align*}
		which completes the proof.
	\end{proof}
	\subsection{Proof of \Cref{prop:error_est}}\label{proof:prop:error_est}
	\begin{proof}
		It is sufficient to prove for $(1,1)\leqslant(i,j)\leqslant(q,d_{q})$ that
		\begin{align}\label{ine:vecvec}
			\left|\eta_{h_{n}}(E_{h_{n}}^{(i)}u_{ij}, \Omega)-\tilde{\eta}_{h_{n}}(\tilde{E}_{n}^{(i)}u_{ij}, \Omega)\right|^{2}\leqslant\tilde{C}\left(\left\Vert E_{h_{n}}^{(i)}u_{ij}-\tilde{E}_{n}^{(i)}u_{ij} \right\Vert_{a}^{2}+\sum_{r=1}^{q}\sum_{s=1}^{d_{r}}\left|\lambda_{rs}^{h_{n}}-\lambda_{rs}^{(n)}\right|^{2}\right), 
		\end{align}
		where $\tilde{C}>0 $ is independent of the mesh size. By the triangle inequality, indeed, there holds
		\begin{equation}\label{ine:func_vec}
			\begin{aligned}
				&\left|\eta_{h_{n}}(\mathcal{E}_{h_{n}}U, \Omega)-\tilde{\eta}_{h_{n}}(\tilde{\mathcal{E}}_{n}U, \Omega)\right|^{2}\\=&\left|\sqrt{\sum_{i=1}^{q}\sum_{j=1}^{d_{i}}\eta_{h_{n}}^{2}(E_{h_{n}}^{(i)}u_{ij}, \Omega)}-\sqrt{\sum_{i=1}^{q}\sum_{j=1}^{d_{i}}\tilde{\eta}_{h_{n}}^{2}(\tilde{E}_{n}^{(i)}u_{ij}, \Omega)}\right|^{2}\\\leqslant&\sum_{i=1}^{q}\sum_{j=1}^{d_{i}}\left|\eta_{h_{n}}(E_{h_{n}}^{(i)}u_{ij}, \Omega)-\tilde{\eta}_{h_{n}}(\tilde{E}_{n}^{(i)}u_{ij}, \Omega)\right|^{2}\\\leqslant&\sum_{i=1}^{q}\sum_{j=1}^{d_{i}}\tilde{C}\left(\left\Vert E_{h_{n}}^{(i)}u_{ij}-\tilde{E}_{n}^{(i)}u_{ij} \right\Vert_{a}^{2}+\sum_{r=1}^{q}\sum_{s=1}^{d_{r}}\left|\lambda_{rs}^{h_{n}}-\lambda_{rs}^{(n)}\right|^{2}\right).
			\end{aligned}
		\end{equation}
		Since $\tilde{E}_{n}^{(i)}:=\mathcal{P}^{(i)}\circ E_{h_{n}}^{(i)}$ and $\mathcal{P}^{(i)}$ is an orthogonal projection with respect to $a(\cdot, \cdot)$ and $h_{n}\leqslant h_{0}\ll1$, we obtain
		\begin{equation}
			\begin{aligned}\label{ine:func_veccccc}
				&\left|\eta_{h_{n}}(\mathcal{E}_{h_{n}}U, \Omega)-\tilde{\eta}_{h_{n}}(\tilde{\mathcal{E}}_{n}U, \Omega)\right|^{2}\\\leqslant&\sum_{i=1}^{q}\sum_{j=1}^{d_{i}}\tilde{C}\left(\left\Vert \left(\mathrm{I}-\mathcal{P}^{(i)}\right)E_{h_{n}}^{(i)}u_{ij}\right\Vert_{a}^{2}+\sum_{r=1}^{q}\sum_{s=1}^{d_{r}}\left|\lambda_{rs}^{h_{n}}-\lambda_{rs}^{(n)}\right|^{2}\right)\\\leqslant&\sum_{i=1}^{q}\sum_{j=1}^{d_{i}}C\left(\operatorname{dist}^{2}_{a}\left(E_{h_{n}}^{(i)}u_{ij}, \tilde{M}_{n}(\lambda_{i})\right)+\sum_{r=1}^{q}\sum_{s=1}^{d_{r}}\left|\lambda_{rs}^{h_{n}}-\lambda_{rs}^{(n)}\right|^{2}\right)\\\leqslant&C\sum_{i=1}^{q}\left(\operatorname{dist}^{2}_{a}\left(M_{h_{n}}(\lambda_{i}), \tilde{M}_{n}(\lambda_{i})\right)+\sum_{j=1}^{d_{i}}\left|\lambda_{ij}^{h_{n}}-\lambda_{ij}^{(n)}\right|^{2}\right),
			\end{aligned}    
		\end{equation}
		where $C=\lambda_{q}\tilde{C}\geqslant\left\Vert u_{ij}\right\Vert_{a}^{2}\tilde{C}\geqslant\left\Vert E_{h_{n}}^{(i)}u_{ij}\right\Vert_{a}^{2}\tilde{C}$ is independent of the mesh size.
		
		To prove \cref{ine:vecvec}, we set $\alpha_{rs}=b\left(E_{h_{n}}^{(i)}u_{ij},u_{rs}^{h_{n}}\right)$ and $\beta_{rs}=b\left(\tilde{E}_{n}^{(i)}u_{ij},u_{rs}^{(n)}\right)$, where $(1,1)\leqslant(r,s)\leqslant(q,d_{q})$. First, we estimate the element residual part as follows
		\begin{equation}\label{ine:resti}
			\begin{aligned}
				&\sum_{T\in\mathscr{T}_{h_{n}}}h_{T}^{2}\left\Vert\mathscr{R}_{T}(E_{h_{n}}^{(i)}u_{ij})-\tilde{\mathscr{R}}_{T}(\tilde{E}_{n}^{(i)}u_{ij}) \right\Vert_{0,T}^{2}\\=&\sum_{T\in\mathscr{T}_{h_{n}}}h_{T}^{2}\left\Vert\sum_{r=1}^{q}\sum_{s=1}^{d_{r}}\alpha_{rs}\lambda_{rs}^{h_{n}}u_{rs}^{h_{n}}+\mathcal{L}E_{h_{n}}^{(i)}u_{ij}-\sum_{r=1}^{q}\sum_{s=1}^{d_{r}}\beta_{rs}\lambda_{rs}^{(n)}u_{rs}^{(n)}-\mathcal{L}\tilde{E}_{n}^{(i)}u_{ij}\right\Vert_{0,T}^{2}\\\leqslant&\sum_{T\in\mathscr{T}_{h_{n}}}h_{T}^{2}\left(\left\Vert\sum_{r=1}^{q}\sum_{s=1}^{d_{r}}\alpha_{rs}\lambda_{rs}^{h_{n}}u_{rs}^{h_{n}}-\sum_{r=1}^{q}\sum_{s=1}^{d_{r}}\beta_{rs}\lambda_{rs}^{(n)}u_{rs}^{(n)}\right\Vert_{0,T}+\left\Vert\mathcal{L}E_{h_{n}}^{(i)}u_{ij}-\mathcal{L}\tilde{E}_{n}^{(i)}u_{ij} \right\Vert_{0,T}\right)^{2}\\\lesssim&\sum_{T\in\mathscr{T}_{h_{n}}}h_{T}^{2}\left(\left\Vert\sum_{r=1}^{q}\sum_{s=1}^{d_{r}}\alpha_{rs}\lambda_{rs}^{h_{n}}u_{rs}^{h_{n}}-\sum_{r=1}^{q}\sum_{s=1}^{d_{r}}\beta_{rs}\lambda_{rs}^{(n)}u_{rs}^{(n)}\right\Vert_{0,T}^{2}+\left\Vert\mathcal{L}E_{h_{n}}^{(i)}u_{ij}-\mathcal{L}\tilde{E}_{n}^{(i)}u_{ij} \right\Vert_{0,T}^{2}\right).
			\end{aligned}
		\end{equation}
		Note that 
		\begin{equation}\label{ine:firrrrrr}
			\begin{aligned}
				&\sum_{T\in\mathscr{T}_{h_{n}}}h_{T}^{2}\left\Vert\sum_{r=1}^{q}\sum_{s=1}^{d_{r}}\alpha_{rs}\lambda_{rs}^{h_{n}}u_{rs}^{h_{n}}-\sum_{r=1}^{q}\sum_{s=1}^{d_{r}}\beta_{rs}\lambda_{rs}^{(n)}u_{rs}^{(n)}\right\Vert_{0,T}^{2}\\\leqslant&h_{0}^{2}\sum_{T\in\mathscr{T}_{h_{n}}}\left\Vert\sum_{r=1}^{q}\sum_{s=1}^{d_{r}}\alpha_{rs}\lambda_{rs}^{h_{n}}u_{rs}^{h_{n}}-\sum_{r=1}^{q}\sum_{s=1}^{d_{r}}\beta_{rs}\lambda_{rs}^{(n)}u_{rs}^{(n)}\right\Vert_{0,T}^{2}\\\lesssim&\left\Vert\sum_{r=1}^{q}\sum_{s=1}^{d_{r}}\lambda_{rs}^{h_{n}}\left(\alpha_{rs}u_{rs}^{h_{n}}-\beta_{rs}u_{rs}^{(n)}\right)\right\Vert_{b}^{2}+\left\Vert\sum_{r=1}^{q}\sum_{s=1}^{d_{r}}\left(\lambda_{rs}^{h_{n}}-\lambda_{rs}^{(n)}\right)\beta_{rs}u_{rs}^{(n)}\right\Vert_{b}^{2}.
			\end{aligned}
		\end{equation}
		Due to $b(u_{ij}^{h_{n}}, u_{kl}^{h_{n}})=b(u_{ij}^{(n)}, u_{kl}^{(n)})=\delta_{ik}\delta_{jl}$, we have
		\begin{equation}\label{eq:first_part}
			\begin{aligned}
				&\left\Vert\sum_{r=1}^{q}\sum_{s=1}^{d_{r}}\lambda_{rs}^{h_{n}}\left(\alpha_{rs}u_{rs}^{h_{n}}-\beta_{rs}u_{rs}^{(n)}\right)\right\Vert_{b}^{2}\\=&\left\Vert\sum_{r=1}^{q}\sum_{s=1}^{d_{r}}\lambda_{rs}^{h_{n}}\alpha_{rs}u_{rs}^{h_{n}} \right\Vert_{b}^{2}+\left\Vert\sum_{r=1}^{q}\sum_{s=1}^{d_{r}}\lambda_{rs}^{h_{n}}\beta_{rs}u_{rs}^{(n)} \right\Vert_{b}^{2}-2\sum_{(r,s)}\sum_{(k,l)}\lambda_{rs}^{h_{n}}\lambda_{kl}^{h_{n}}\alpha_{rs}\beta_{kl}b\left(u_{rs}^{h_{n}}, u_{kl}^{(n)}\right)\\=&\sum_{r=1}^{q}\sum_{s=1}^{d_{r}}\left(\left(\lambda_{rs}^{h_{n}}\alpha_{rs}\right)^{2}+\left(\lambda_{rs}^{h_{n}}\beta_{rs}\right)^{2}\right)-2\sum_{(r,s)}\sum_{(k,l)}\lambda_{kl}^{h_{n}}\alpha_{rs}\beta_{kl}a\left(u_{rs}^{h_{n}}, u_{kl}^{(n)}\right),
			\end{aligned}
		\end{equation}
		and 
		\begin{equation}\label{eq:a_norm}
			\begin{aligned}
				&\left\Vert E_{h_{n}}^{(i)}u_{ij}-\tilde{E}_{n}^{(i)}u_{ij} \right\Vert_{a}^{2}\\=&\left\Vert\sum_{r=1}^{q}\sum_{s=1}^{d_{r}}\alpha_{rs}u_{rs}^{h_{n}}-\sum_{r=1}^{q}\sum_{s=1}^{d_{r}}\beta_{rs}u_{rs}^{(n)}\right\Vert_{a}^{2}\\=&\sum_{r=1}^{q}\sum_{s=1}^{d_{r}}\left(\lambda_{rs}^{h_{n}}\alpha_{rs}^{2}+\lambda_{rs}^{(n)}\beta_{rs}^{2}\right)-2\sum_{(r,s)}\sum_{(k,l)}\alpha_{rs}\beta_{kl}a\left(u_{rs}^{h_{n}}, u_{kl}^{(n)}\right).
			\end{aligned}
		\end{equation}
		Since $h_{n}\leqslant h_{0}\ll1$, there holds that $\lambda_{q+1}\geqslant\lambda_{rs}^{h_{n}}$ for $(1,1)\leqslant(r,s)\leqslant(q,d_{q})$. 
		Multiply both sides of equation \cref{eq:a_norm} by $\hat{C}\geqslant\lambda_{q+1}$ and subtract \cref{eq:first_part}, then we arrive at
		\begin{equation*}
			\begin{aligned}
				&\hat{C}\left\Vert E_{h_{n}}^{(i)}u_{ij}-\tilde{E}_{n}^{(i)}u_{ij} \right\Vert_{a}^{2}-\left\Vert\sum_{r=1}^{q}\sum_{s=1}^{d_{r}}\lambda_{rs}^{h_{n}}\left(\alpha_{rs}u_{rs}^{h_{n}}-\beta_{rs}u_{rs}^{(n)}\right)\right\Vert_{b}^{2}\\=&\sum_{r=1}^{q}\sum_{s=1}^{d_{r}}\left(\left(\hat{C}-\lambda_{rs}^{h_{n}}\right)\lambda_{rs}^{h_{n}}\alpha_{rs}^{2}+\left(\hat{C}\lambda_{rs}^{(n)}-\left(\lambda_{rs}^{h_{n}}\right)^{2}\right)\beta_{rs}^{2}\right)\\&-2\sum_{(r,s),(k,l)}\left(\hat{C}-\lambda_{kl}^{h_{n}}\right)\alpha_{rs}\beta_{kl}a\left(u_{rs}^{h_{n}}, u_{kl}^{(n)}\right)\\=&\sum_{r=1}^{q}\sum_{s=1}^{d_{r}}\left\Vert\sqrt{\hat{C}-\lambda_{rs}^{h_{n}}}\left(\alpha_{rs}u_{rs}^{h_{n}}-\beta_{rs}u_{rs}^{(n)}\right)\right\Vert_{a}^{2}+\sum_{r=1}^{q}\sum_{s=1}^{d_{r}}\lambda_{rs}^{h_{n}}\left(\lambda_{rs}^{(n)}-\lambda_{rs}^{h_{n}}\right)\beta_{rs}^{2}\geqslant0,
			\end{aligned}
		\end{equation*}
		since for $(1,1)\leqslant(r,s)\leqslant(q,d_{q})$, $\left\Vert u_{rs}^{h_{n}}\right\Vert^{2}_{a}=\lambda_{rs}^{h_{n}}$ and $\left\Vert u_{rs}^{(n)}\right\Vert^{2}_{a}=\lambda_{rs}^{(n)}$, and $\lambda_{rs}^{(n)}\geqslant\lambda_{rs}^{h_{n}}$ by the minimum-maximum principle. Namely, we get
		\begin{align}\label{ine:pere}
			\left\Vert\sum_{r=1}^{q}\sum_{s=1}^{d_{r}}\lambda_{rs}^{h_{n}}\left(\alpha_{rs}u_{rs}^{h_{n}}-\beta_{rs}u_{rs}^{(n)}\right)\right\Vert_{b}^{2}\lesssim \left\Vert E_{h_{n}}^{(i)}u_{ij}-\tilde{E}_{n}^{(i)}u_{ij} \right\Vert_{a}^{2}.
		\end{align}
		Subsequently, it follows from the fact
		\begin{align*}
			\sum_{r=1}^{q}\sum_{s=1}^{d_{r}}\beta_{rs}^{2}=\left\Vert \tilde{E}_{n}^{(i)}u_{ij}\right\Vert_{b}^{2}\lesssim \left\Vert \tilde{E}_{n}^{(i)}u_{ij}\right\Vert_{a}^{2}\leqslant\left\Vert u_{ij}\right\Vert_{a}^{2}=\lambda_{i}
		\end{align*}
		that
		\begin{align*}
			\left\Vert\sum_{r=1}^{q}\sum_{s=1}^{d_{r}}\left(\lambda_{rs}^{h_{n}}-\lambda_{rs}^{(n)}\right)\beta_{rs}u_{rs}^{(n)}\right\Vert_{b}^{2}=\sum_{r=1}^{q}\sum_{s=1}^{d_{r}}\left(\lambda_{rs}^{h_{n}}-\lambda_{rs}^{(n)}\right)^{2}\beta_{rs}^{2}\lesssim\sum_{r=1}^{q}\sum_{s=1}^{d_{r}}\left|\lambda_{rs}^{h_{n}}-\lambda_{rs}^{(n)}\right|^{2},
		\end{align*}
		which together with \cref{ine:firrrrrr} and \cref{ine:pere} yields that
		\begin{equation}\label{ine:resti1111}
			\begin{aligned}
				&\sum_{T\in\mathscr{T}_{h_{n}}}h_{T}^{2}\left\Vert\sum_{r=1}^{q}\sum_{s=1}^{d_{r}}\alpha_{rs}\lambda_{rs}^{h_{n}}u_{rs}^{h_{n}}-\sum_{r=1}^{q}\sum_{s=1}^{d_{r}}\beta_{rs}\lambda_{rs}^{(n)}u_{rs}^{(n)}\right\Vert_{0,T}^{2}\\\lesssim& \left\Vert E_{h_{n}}^{(i)}u_{ij}-\tilde{E}_{n}^{(i)}u_{ij} \right\Vert_{a}^{2}+\sum_{i=1}^{q}\sum_{s=1}^{d_{r}}\left|\lambda_{rs}^{h_{n}}-\lambda_{rs}^{(n)}\right|^{2}.
			\end{aligned}
		\end{equation}
		For the second part of \cref{ine:resti}, we obtain from the inverse inequality that
		\begin{align*}
			&\sum_{T\in\mathscr{T}_{h_{n}}}h_{T}^{2}\left\Vert\mathcal{L}E_{h_{n}}^{(i)}u_{ij}-\mathcal{L}\tilde{E}_{n}^{(i)}u_{ij} \right\Vert_{0,T}^{2}\\\lesssim&\sum_{T\in\mathscr{T}_{h_{n}}}h_{T}^{2}\left\Vert\nabla\cdot\left(A\nabla\left(E_{h_{n}}^{(i)}u_{ij}-\tilde{E}_{n}^{(i)}u_{ij}\right)\right)\right\Vert_{0,T}^{2}+\sum_{T\in\mathscr{T}_{h_{n}}}h_{T}^{2}\left\Vert c\left(E_{h_{n}}^{(i)}u_{ij}-\tilde{E}_{n}^{(i)}u_{ij}\right)\right\Vert_{0,T}
			^{2}\\\lesssim&\sum_{T\in\mathscr{T}_{h_{n}}}\left\Vert A\nabla\left(E_{h_{n}}^{(i)}u_{ij}-\tilde{E}_{n}^{(i)}u_{ij}\right)\right\Vert_{0,T}^{2}+\left\Vert E_{h_{n}}^{(i)}u_{ij}-\tilde{E}_{n}^{(i)}u_{ij}\right\Vert_{b}^{2}
			\\\lesssim&\left\Vert E_{h_{n}}^{(i)}u_{ij}-\tilde{E}_{n}^{(i)}u_{ij}\right\Vert_{a}^{2},
		\end{align*}
		and conclude from \cref{ine:resti} and \cref{ine:resti1111} that
		\begin{equation}\label{ine:res_part}
			\begin{aligned}
				&\sum_{T\in\mathscr{T}_{h_{n}}}h_{T}^{2}\left\Vert\mathscr{R}_{T}(E_{h_{n}}^{(i)}u_{ij})-\tilde{\mathscr{R}}_{T}(\tilde{E}_{n}^{(i)}u_{ij}) \right\Vert_{0,T}^{2}\\\lesssim&\left\Vert E_{h_{n}}^{(i)}u_{ij}-\tilde{E}_{n}^{(i)}u_{ij}\right\Vert_{a}^{2}+\sum_{r=1}^{q}\sum_{s=1}^{d_{r}}\left|\lambda_{rs}^{h_{n}}-\lambda_{rs}^{(n)}\right|^{2}.
			\end{aligned}  
		\end{equation}
		
		Second, for the jump residual part, we obtain from the trace inequality and the inverse inequality that
		\begin{equation}\label{ine:jum_part}
			\begin{aligned}
				&\sum_{T\in\mathscr{T}_{h_{n}}}\sum_{e\in\mathscr{E}_{h_{n}}, e\subset\partial T}h_{e}\left\Vert\mathscr{J}_{e}(E_{h_{n}}^{(i)}u_{ij})-\tilde{\mathscr{J}}_{e}(\tilde{E}_{n}^{(i)}u_{ij})\right\Vert_{0,e}^{2}\\=&\sum_{T\in\mathscr{T}_{h_{n}}}\sum_{e\in\mathscr{E}_{h_{n}}, e\subset\partial T}h_{e}\left\Vert \llbracket A \nabla \left(E_{h_{n}}^{(i)}u_{ij}-\tilde{E}_{n}^{(i)}u_{ij}\right) \rrbracket_e \cdot \nu_e\right\Vert_{0,e}^{2}\\\lesssim&\sum_{T\in\mathscr{T}_{h_{n}}}\sum_{e\in\mathscr{E}_{h_{n}}, e\subset\partial T}h_{e}\left( h_{e}^{-1} \left\Vert A \nabla \left( E_{h_{n}}^{(i)}u_{ij} - \tilde{E}_{n}^{(i)}u_{ij} \right) \right\Vert_{L^{2}(w_{e})}^{2} \right. \\
				&\left. + h_{e} \left\Vert \nabla\cdot \left( A\nabla \left( E_{h_{n}}^{(i)}u_{ij} - \tilde{E}_{n}^{(i)}u_{ij} \right) \right) \right\Vert_{L^{2}(w_{e})}^{2} \right)\\\lesssim&\sum_{T\in\mathscr{T}_{h_{n}}}\sum_{e\in\mathscr{E}_{h_{n}}, e\subset\partial T}\left\Vert A \nabla \left( E_{h_{n}}^{(i)}u_{ij} - \tilde{E}_{n}^{(i)}u_{ij} \right) \right\Vert_{L^{2}(w_{e})}^{2}\\\lesssim&\left\Vert E_{h_{n}}^{(i)}u_{ij} - \tilde{E}_{n}^{(i)}u_{ij}\right\Vert_{a}^{2},
			\end{aligned}
		\end{equation}
		where $\omega(e):=T^{+}\cup T^{-}$ for $T^{+}, T^{-}\in \mathscr{T}_{h_{n}}$ $e=T^{+}\cap T^{-}$.
		
		In summary, combing with the triangle inequality, \cref{ine:res_part} and \cref{ine:jum_part}, we have 
		\begin{equation}\label{ine:summm}
			\begin{aligned}
				&\left|\eta_{h_{n}}(E_{h_{n}}^{(i)}u_{ij}, \Omega)-\tilde{\eta}_{h_{n}}(\tilde{E}_{n}^{(i)}u_{ij}, \Omega)\right|^{2}\\\leqslant&\sum_{T\in\mathscr{T}_{h_{n}}}\left(h_{T}^{2}\left|\left\Vert\mathscr{R}_{T}(E_{h_{n}}^{(i)}u_{ij})\right\Vert_{0,T}-\left\Vert\tilde{\mathscr{R}}_{T}(\tilde{E}_{n}^{(i)}u_{ij}) \right\Vert_{0,T}\right|^{2}\right.\\
				&\left.+\sum_{e\in\mathscr{E}_{h_{n}}, e\subset\partial T}h_{e}\left|\left\Vert\mathscr{J}_{e}(E_{h_{n}}^{(i)}u_{ij})\right\Vert_{0,e}-\left\Vert\tilde{\mathscr{J}}_{e}(\tilde{E}_{n}^{(i)}u_{ij}) \right\Vert_{0,e}\right|^{2}\right)\\\leqslant&\sum_{T\in\mathscr{T}_{h_{n}}}\left(h_{T}^{2}\left\Vert\mathscr{R}_{T}(E_{h_{n}}^{(i)}u_{ij})-\tilde{\mathscr{R}}_{T}(\tilde{E}_{n}^{(i)}u_{ij}) \right\Vert_{0,T}^{2}\right.\\
				&\left.+\sum_{e\in\mathscr{E}_{h_{n}}, e\subset\partial T}h_{e}\left\Vert\mathscr{J}_{e}(E_{h_{n}}^{(i)}u_{ij})-\tilde{\mathscr{J}}_{e}(\tilde{E}_{n}^{(i)}u_{ij})\right\Vert_{0,e}^{2}\right)\\\lesssim&\left\Vert E_{h_{n}}^{(i)}u_{ij}-\tilde{E}_{n}^{(i)}u_{ij}\right\Vert_{a}^{2}+\sum_{r=1}^{q}\sum_{s=1}^{d_{r}}\left|\lambda_{rs}^{h_{n}}-\lambda_{rs}^{(n)}\right|^{2},
			\end{aligned}
		\end{equation}
		which completes the proof.
	\end{proof}
	\subsection{Proof of \Cref{thm:dof}}\label{proof:thm:dof}
	\begin{proof}
		We obtain from the triangle inequality that
		\begin{equation*}
			\begin{aligned}
				&\sum_{T\in\mathscr{T}_{h_{n}}}\sum_{i=1}^{q}\sum_{j=1}^{d_{i}}\left|\eta_{h_{n}}(E_{h_{n}}^{(i)}u_{ij}, T)-\tilde{\eta}_{h_{n}}(\tilde{E}_{n}^{(i)}u_{ij}, T)\right|^{2}\\\leqslant&\sum_{i=1}^{q}\sum_{j=1}^{d_{i}}\sum_{T\in\mathscr{T}_{h_{n}}}\left(h_{T}^{2}\left|\left\Vert\mathscr{R}_{T}(E_{h_{n}}^{(i)}u_{ij})\right\Vert_{0,T}-\left\Vert\tilde{\mathscr{R}}_{T}(\tilde{E}_{n}^{(i)}u_{ij}) \right\Vert_{0,T}\right|^{2}\right.\\
				&\left.+\sum_{e\in\mathscr{E}_{h_{n}}, e\subset\partial T}h_{e}\left|\left\Vert\mathscr{J}_{e}(E_{h_{n}}^{(i)}u_{ij})\right\Vert_{0,e}-\left\Vert\tilde{\mathscr{J}}_{e}(\tilde{E}_{n}^{(i)}u_{ij}) \right\Vert_{0,e}\right|^{2}\right),
			\end{aligned}
		\end{equation*}
		which together with \cref{ine:func_veccccc}, \cref{ine:res_part}, \cref{ine:jum_part} and \cref{ine:approxxxx} yields that
		\begin{equation}\label{ine:estimatttt}
			\begin{aligned}
				&\sum_{T\in\mathscr{T}_{h_{n}}}\sum_{i=1}^{q}\sum_{j=1}^{d_{i}}\left|\eta_{h_{n}}(E_{h_{n}}^{(i)}u_{ij}, T)-\tilde{\eta}_{h_{n}}(\tilde{E}_{n}^{(i)}u_{ij}, T)\right|^{2}\\\leqslant&\sum_{i=1}^{q}\sum_{j=1}^{d_{i}}\tilde{C}\left(\left\Vert E_{h_{n}}^{(i)}u_{ij}-\tilde{E}_{n}^{(i)}u_{ij} \right\Vert_{a}^{2}+\sum_{r=1}^{q}\sum_{s=1}^{d_{r}}\left|\lambda_{rs}^{h_{n}}-\lambda_{rs}^{(n)}\right|^{2}\right)\\\leqslant& C\sum_{i=1}^{q}\left(\operatorname{dist}^{2}_{a}\left(M_{h_{n}}(\lambda_{i}), \tilde{M}_{n}(\lambda_{i})\right)+\sum_{j=1}^{d_{i}}\left|\lambda_{ij}^{h_{n}}-\lambda_{ij}^{(n)}\right|^{2}\right)\leqslant C\rho\tilde{\eta}_{h_{n}}^{2}\left(\tilde{\mathcal{E}}_{n}U, \Omega\right).
			\end{aligned}
		\end{equation}
		
		We see from \Cref{prop:error_est} and \cref{ine:estimatttt} that
		\begin{equation}
			\begin{aligned}\label{ine:esssndh}
				\eta_{h_{n}}^{2}(\mathcal{E}_{h_{n}}U,\Omega)\leqslant&\left(\tilde{\eta}_{h_{n}}(\tilde{\mathcal{E}}_{n}U,\Omega)+\sqrt{C\sum_{i=1}^{q}\left(\operatorname{dist}^{2}_{a}\left(M_{h_{n}}(\lambda_{i}), \tilde{M}_{n}(\lambda_{i})\right)+\sum_{j=1}^{d_{i}}\left|\lambda_{ij}^{h_{n}}-\lambda_{ij}^{(n)}\right|^{2}\right)}\right)^{2}\\\leqslant&\left(1+\sqrt{C\rho}\right)^{2}\tilde{\eta}_{h_{n}}^{2}(\tilde{\mathcal{E}}_{n}U,\Omega).
			\end{aligned}
		\end{equation}
		We derive from the triangle inequality and Cauchy-Schwarz inequality that
		\begin{align*}
			&\left|\sum_{T \in \hat{\mathscr{T}}_{h_{n}}}\eta_{h_{n}}^{2}(\mathcal{E}_{h_{n}}U,T)-\sum_{T \in \hat{\mathscr{T}}_{h_{n}}}\tilde{\eta}_{h_{n}}^{2}(\tilde{\mathcal{E}}_{n}U,T)\right|\leqslant\sum_{T \in \hat{\mathscr{T}}_{h_{n}}}\left|\eta_{h_{n}}^{2}(\mathcal{E}_{h_{n}}U,T)-\tilde{\eta}_{h_{n}}^{2}(\tilde{\mathcal{E}}_{n}U,T)\right|\\\leqslant&\sum_{T\in\mathscr{T}_{h_{n}}}\sum_{i=1}^{q}\sum_{j=1}^{d_{i}}\left|\eta_{h_{n}}(E_{h_{n}}^{(i)}u_{ij}, T)+\tilde{\eta}_{h_{n}}(\tilde{E}_{n}^{(i)}u_{ij}, T)\right|\left|\eta_{h_{n}}(E_{h_{n}}^{(i)}u_{ij}, T)-\tilde{\eta}_{h_{n}}(\tilde{E}_{n}^{(i)}u_{ij}, T)\right|\\\leqslant&\sqrt{\left(\sum_{T\in\mathscr{T}_{h_{n}}}\sum_{i=1}^{q}\sum_{j=1}^{d_{i}}\left|\eta_{h_{n}}(E_{h_{n}}^{(i)}u_{ij}, T)+\tilde{\eta}_{h_{n}}(\tilde{E}_{n}^{(i)}u_{ij}, T)\right|^{2}\right)}\\&\times\sqrt{\left(\sum_{T\in\mathscr{T}_{h_{n}}}\sum_{i=1}^{q}\sum_{j=1}^{d_{i}}\left|\eta_{h_{n}}(E_{h_{n}}^{(i)}u_{ij}, T)-\tilde{\eta}_{h_{n}}(\tilde{E}_{n}^{(i)}u_{ij}, T)\right|^{2}\right)}\\\leqslant&\sqrt{2\left(\eta_{h_{n}}^{2}(\mathcal{E}_{h_{n}}U,\Omega)+\tilde{\eta}_{h_{n}}^{2}(\tilde{\mathcal{E}}_{n}U,\Omega)\right)C\rho\tilde{\eta}_{h_{n}}^{2}\left(\tilde{\mathcal{E}}_{n}U, \Omega\right)},
		\end{align*}
		which together with \cref{ine:esssndh} leads to
		\begin{align*}
		    &\left|\sum_{T \in \hat{\mathscr{T}}_{h_{n}}}\eta_{h_{n}}^{2}(\mathcal{E}_{h_{n}}U,T)-\sum_{T \in \hat{\mathscr{T}}_{h_{n}}}\tilde{\eta}_{h_{n}}^{2}(\tilde{\mathcal{E}}_{n}U,T)\right|\\\leqslant&\sqrt{2\left(C\rho+2\sqrt{C\rho}+2\right)C\rho}\tilde{\eta}_{h_{n}}^{2}\left(\tilde{\mathcal{E}}_{n}U, \Omega\right)
		\end{align*}
		Hence, we obtain from \cref{ine:dor_eta} and \cref{ine:esssndh} that 
		\begin{align*}
			\sum_{T \in \hat{\mathscr{T}}_{h_{n}}}\eta_{h_{n}}^{2}(\mathcal{E}_{h_{n}}U,T)\geqslant&\sum_{T \in \hat{\mathscr{T}}_{h_{n}}}\tilde{\eta}_{h_{n}}^{2}(\tilde{\mathcal{E}}_{n}U,T)-\sqrt{2\left(C\rho+2\sqrt{C\rho}+2\right)C\rho}\tilde{\eta}_{h_{n}}^{2}\left(\tilde{\mathcal{E}}_{n}U, \Omega\right)\\\geqslant&\left(\theta-\sqrt{2\left(C\rho+2\sqrt{C\rho}+2\right)C\rho}\right)\tilde{\eta}_{h_{n}}^{2}\left(\tilde{\mathcal{E}}_{n}U, \Omega\right)\\\geqslant&\frac{\theta-\sqrt{2\left(C\rho+2\sqrt{C\rho}+2\right)C\rho}}{\left(1+\sqrt{C\rho}\right)^{2}}\eta_{h_{n}}^{2}(\mathcal{E}_{h_{n}}U,\Omega)\geqslant\tilde{\theta}\eta_{h_{n}}^{2}(\mathcal{E}_{h_{n}}U,\Omega),
		\end{align*}
		which completes the proof.
	\end{proof}

	\bibliographystyle{siamplain}
	\bibliography{references}
\end{document}

%% file: ex_shared.tex
\usepackage{lipsum}
\usepackage{amsfonts}
\usepackage{graphicx}
\usepackage{epstopdf}
\usepackage{enumitem}
\usepackage{amsmath} 
\usepackage{amssymb}
\usepackage{stmaryrd}
\usepackage{mathtools}
\usepackage{mathrsfs}
\usepackage{mdframed}
\ifpdf
  \DeclareGraphicsExtensions{.eps,.pdf,.png,.jpg}
\else
  \DeclareGraphicsExtensions{.eps}
\fi

\newsiamremark{remark}{Remark}
\newsiamremark{hypothesis}{Hypothesis}
\newsiamremark{example}{Example}
\crefname{hypothesis}{Hypothesis}{Hypotheses}
\newsiamthm{claim}{Claim}

\title{Convergence of the adaptive finite element discretization based parallel orbital-updating method for eigenvalue problems\thanks{Submitted to the editors DATE.
\funding{This work was supported by the National Natural Science Foundation of China under grants 92270206 and  12571446, the National Key R \& D Program of China under grants 2019YFA0709600 and 2019YFA0709601, and the Strategic Priority Research Program of the Chinese Academy of Sciences under grand XDB0640000. B. Yang also acknowledges the support from the Fundamental Research Funds for the Central Universities and the disciplinary funding of Central University of Finance and Economics.}}}

\author{Xiaoying Dai\thanks{SKLMS, Academy of Mathematics and Systems Science, Chinese Academy of Sciences, Beijing 100190, China; and School of Mathematical Sciences, University of Chinese Academy of Sciences, Beijing 100049, China (\email{daixy@lsec.cc.ac.cn}, \email{liyan2021@lsec.cc.ac.cn}, \email{azhou@lsec.cc.ac.cn}).} \and
Yan Li\footnotemark[2]
\and Bin Yang\thanks{School of Statistics and Mathematics, Central University of Finance and Economics, Beijing 102206, China (\email{binyang@lsec.cc.ac.cn}).}
\and Aihui Zhou\footnotemark[2]}

\usepackage{amsopn}

\usepackage[version=4]{mhchem}
\usepackage{algorithm}
\usepackage{algpseudocode}
